\documentclass[11pt]{article}

\usepackage{amsmath,mathrsfs,bbding}
\usepackage{amssymb}
\usepackage{amsthm}
\usepackage{appendix}
\usepackage{authblk}
\usepackage{caption}
\usepackage{cases}
\usepackage[margin=2cm]{geometry}
\usepackage{graphicx} 
\usepackage[utf8]{inputenc}
\usepackage{mathtools}
\usepackage{listings}
\usepackage{lscape}
\usepackage{pifont}
\usepackage{tikz,tikz-cd}
\usepackage{titlesec}
\usepackage{xcolor}

\usepackage[
backend=biber,
style=alphabetic,
maxbibnames=99,
maxalphanames=10,     
minalphanames=10     
]{biblatex}
\AtEveryBibitem{\clearfield{note}}
\AtEveryBibitem{\clearfield{url}}
\AtEveryBibitem{\clearfield{issn}}
\AtEveryBibitem{\clearfield{doi}}
\AtEveryBibitem{\clearfield{eprint}}
\DeclareFieldFormat[article,periodical]{volume}{Vol.~#1}
\newtheorem{proposition}{Proposition}[section]
\newtheorem{theorem}[proposition]{Theorem}
\newtheorem{corollary}[proposition]{Corollary}

\newtheorem{lemma}[proposition]{Lemma}
\newtheorem{lemma-definition}[proposition]{Lemma-Definition}
  \theoremstyle{definition}
  \newtheorem{definition}[proposition]{Definition}
  \newtheorem{remark}[proposition]{\it Remark}
  \newtheorem{example}[proposition]{Example}

\newtheorem{maintheorem}{Theorem}
\newtheorem{maincorollary}[maintheorem]{Corollary}

\newcommand{\Q}{{\mathbb Q}}    \newcommand{\bb}[1]{\mathbb #1}

\titleformat{\section}
{\large\bfseries}
{\thesection}{1em}{}

\titleformat{\subsection}
{\normalsize\bfseries}
{\thesubsection}{1em}{}

\titlespacing*{\section}
{0pt}      
{1.5ex}      
{0.5ex}    
\titlespacing*{\subsection}
{0pt}      
{1ex}      
{0.5ex}    

\makeatletter
\newcommand{\address}[1]{\gdef\@address{#1}}
\newcommand{\email}[1]{\gdef\@email{\url{#1}}}
\newcommand{\@endstuff}{\par\vspace{\baselineskip}\noindent\small
\begin{tabular}{@{}l}\scshape\@address\\\textit{E-mail address:} \@email\end{tabular}}
\AtEndDocument{\@endstuff}
\makeatother
\title{K-stability of a class of smooth Fano complete intersections}
\author{Yijue Hu}
\address{School of Mathematical Sciences, University of Nottingham}
\email{yijue.hu@nottingham.ac.uk}
\usepackage{lipsum}
\date{}

\begin{document}
\maketitle

\begin{abstract}
We develop an approach to study K-stability of Fano varieties by estimating the local stability threshold along a reducible zero-dimensional subscheme on a smooth curve. As an application, we we show that all smooth index-one Fano hypersurfaces in quadrics are K-stable, when the dimension is higher than one. To prove it, we apply the Abban-Zhuang method to reduce the stability problem to lower dimensions. Specifically, we construct suitable admissible flags such that the last term of the Abban-Zhuang estimate is a local stability threshold along multiple points on a smooth curve, and give a convex geometric interpretation of the estimate. 
\end{abstract}

\section{Introduction}
Since its introduction by Mumford in the 1960s, geometric invariant theory (GIT) has become a powerful tool in algebraic geometry, especially in moduli theory. However, the setup of GIT is limited to address broader moduli problems, including Fano varieties in general. The notion of K-stability was then introduced, inspired by the Hilbert--Mumford criterion in GIT, to provide another framework. In particular, K-moduli spaces exist and parametrize K-polystable Fano varieties with certain fixed numerical invariants. We refer the reader to the comprehensive book [Xu25] for details. 

On the other hand, GIT stability, when it can be studied explicitly, often sheds light on the study of K-stability. One may ask whether a GIT stable Fano variety is K-stable, or if we can compare the K-moduli spaces of some family of varieties with its GIT moduli, or more generally, with a variation of GIT (VGIT) quotient. 

These questions have motivated substantial research. For example,
Liu and Xu \cite{LiuXucubicthreefold} showed that the K-moduli of cubic threefolds is identical to its GIT moduli. Based on this, a complete classification of K-(semi/poly)stable cubic threefolds is derived. A follow-up work \cite{liu2022CubicFourfold} showed the same results for cubic fourfolds. However, this does not always hold true. In fact, such an isomorphism is expected only for cubic hypersurfaces. For example, a multiple of a smooth quadric is always GIT-polystable but K-unstable, as it is non-reduced. It can also fail for non-hypersurface Fano varieties, for example, \autocite[Theorem 1.4]{ADL2023QuarticK3Surfaces} gave an explicit analysis of K-moduli and GIT moduli space of quartic double solids.

Concerning Fano hypersurfaces, since every smooth hypersurface with degree $d\geq 3$ is GIT stable, it is then natural to expect them to be K-stable as well. To summarize the known results, we firstly recall the notion of Fano index. For a Fano variety $X$, its Fano index $r(X)$ is defined as the largest integer $r$ such that the anticanonical divisor $-K_X\sim_{\Q}rL$ for some Cartier divisor $L$. For example, for a hypersurface of degree $d$ in $\bb{P}^{n+1}$, its Fano index is simply $n+2-d$ by adjunction. Projective spaces $\bb{P}^n$ and smooth quadric hypersurfaces are known to be strictly K-polystable, and a general smooth Fano hypersurface of degree $d\geq 3$ is K-stable. We refer to \autocite[Corollary 4.17]{ZhuangOptimalDestabilizing} for an algebraic proof. In low degree, \cite{LiuXucubicthreefold} and \cite{LiuCubicFourfold} showed that all smooth cubic threefolds and fourfolds are K-stable, by relating the notions of K-stable and GIT stable, as discussed above. At the opposite end, when the degree is high, which is equivalent to the Fano index being low, Fujita \autocite[Example 1.4]{Fujita_2019} proved that every smooth Fano hypersurface of Fano index 1 is K-stable; and the index 2 case is proved by Abban and Zhuang \autocite[Theorem 4.12]{AZ1}. A subsequent work \cite{AZ2} of the same two authors showed that every smooth Fano hypersurface of Fano index $r\geq 3$ and dimension $n \geq r^3$ is K-stable. This is the strongest result currently known. The K-stability problem for smooth hypersurfaces of intermediate degrees remains open.

Fujita's work considered the alpha invariant of Fano varieties, which gives a sufficient, but not necessary, criterion for K-stability. The delta invariant, also known as the stability threshold, introduced by Fujita and Odaka in \cite{FujitaOdakaLowerAlpha} is better suited to explicit K-stability problems. In \cite{AZ1}, Abban and Zhuang developed a general method to estimate the delta invariant, which is now commonly referred to as the Abban--Zhuang method. The core idea is to reduce the K-stability problem to the lower-dimensional cases. 

As complete intersections are a natural generalization of hypersurfaces, their GIT and K-stability properties are also of interest. They are less well understood in general, and the connection between them is not as close. On the GIT side, Benoist \cite{BenoistCompleteIntersection} discussed GIT stability of complete intersections with various multidegrees and under different parameter spaces. In particular, under a natural parameter space and group action, there always exists a linearization assuring all smooth complete intersections are stable. Moreover, moduli spaces are studied for a few low-degree complete-intersection curves or surfaces. On the K-stability side, it is recently proved in \cite{generalCI} that a general smooth Fano complete intersection, which is not $\bb{P}^n$ or a quadric, is K-stable. Moreover, partial results are known when the degrees of hypersurfaces are low, or the Fano index is low, as one could expect. More specifically, a classical result of \autocite[Corollary 3.1]{Ghigi2004symmetries} shows that every smooth complete intersection of two quadrics in $\bb{P}^n$ with $n>3$ is K-stable, while \autocite[Theorem 5.1]{AZ2} implies that every smooth complete intersection of three quadrics in $\bb{P}^6$ is K-stable. In addition, \autocite[Theorem 1.3]{ZhuangSupperrigidity} showed that every smooth complete intersection $X_{2,3}\subseteq \mathbb{P}^5$ of a quadric and a cubic is K-stable. The second part of the same theorem tells that every smooth complete intersection $X \subseteq \mathbb{P}^{n+r}$ of index 1, codimension $r$, and dimension $n \geq 10r$ is K-stable. In particular, when~$r=2$, $n\geq 12$ is sufficient. 

The proof in \cite{ZhuangSupperrigidity} relies on the birational superrigidity property  of index one Fano complete intersections, which currently is only known in relatively high dimension, and expected for only general ones in lower dimension. To study their K-stability in low dimension, one needs to seek other techniques. The Abban--Zhuang method provides such an approach. To estimate the local stability threshold along a closed subset $Z\subset X$ of a Fano variety $X$, we can pick a suitable admissible flag terminating at a variety $Y$ of lower dimension, and reduce the problem to estimating the delta invariant of a refined multigraded linear series on $Y$ along $Z\cap Y$. The terminologies and estimate are discussed in detail in Section \ref{preliminary}. We apply this method to prove the following theorem. 

\begin{maintheorem}\label{main theorem}
    Let $\mathcal{X}_{2,n}$ denote the family of complete intersections of degree $2$ and degree $n$ hypersurfaces in $\mathbb{P}^{n+2}$. Every smooth member of $\mathcal{X}_{2,n}$ is K-stable for all $n>1$.
\end{maintheorem}

This result fills the gap for a special case in \autocite[Theorem 1.3]{ZhuangSupperrigidity} when $r=2$ and $n$ is small, while our argument treat all $n\geq 3$ uniformly. The main idea of the proof is inspired by the approach of \autocite[\textsection 4.5]{AZ1} for establishing K-stability of Fano hypersurfaces of index 2. Their argument estimated the local stability threshold of such hypersurfaces along an arbitrary curve, by choosing a suitable flag terminating at a one-dimensional subvariety meeting the curve at multiple points. More precisely, for any variety $X$ in our family, and any irreducible curve $C\subset X$, we choose the flag to be as simple as possible, typically using hyperplane sections, while ensuring that each member meets $C$ in at least 2 points, say $P$ and~$Q$. Let the flag terminate at a one-dimensional subvariety $Y$, we then estimate the delta invariant of $Y$ along~$C\cap Y$, which includes at least two points $\{P,Q\}$. Two special cases require careful treatment: the case that the line joining $P$ and $Q$ is contained in $X$, and consequently contained in each flag member of the hyperplane section form, and the case $C$ itself is a line. This will be elaborated in Section \ref{proof of theorem A}. 

The proof strategy highlights the importance of the local stability threshold along multiple points, or more generally, along reducible locus, in explicit K-stability problems. The idea already appears in \cite{AZ1}, but it is not developed systematically there. In Section \ref{Stability Threshold Along Several Points} we study this object in detail.
This analysis leads to a generalization of a key lemma \autocite[Lemma 4.16]{AZ1} to our Lemma \ref{GeneralizedLocalDeltaOnCurve}. 

Moreover, for an admissible flag terminating at a one-dimensional subvariety, we associate to it a convex body closely related to an Okounkov body, a classical convex-geometric object in algebraic geometry. Then we relate the local stability threshold, together with the other terms appearing in the Abban--Zhuang's estimate, to the coordinates of the barycenter of this convex body. Furthermore, we show that, under suitable assumptions on the admissible flag, the problem can be reduced to a lower-dimensional convex body. When the dimension is sufficiently low, the coordinates of its barycenter can be estimated practically. Finally, in Section \ref{proof of theorem A}, we apply these results to prove Theorem \ref{main theorem}.

Although K-stability has become a powerful framework in moduli problems, it was originally introduced to give an algebro-geometric characterization of the existence of K\"ahler--Einstein metrics on Fano manifolds. The celebrated Yau--Tian--Donaldson conjecture, which asserts that a Fano manifold admits a K\"ahler--Einstein metric if and only if it is K-polystable, has been proved through a series of works, see \cite{ChenDnoaldosnSunI}, \cite{ChenDnoaldosnSunII}, \cite{ChenDnoaldosnSunIII}, \cite{TianK-stability}, and \cite{Berman}. More recently, substantial progress has also been made in the more general setting, that is, the existence of a constant scalar curvature K\"ahler metric on a polarised smooth variety, see \cite{trusiani2026YTD}. Therefore we have the following straightforward corollary.

\begin{maincorollary}
    For $n>1$, every smooth member of $\mathcal{X}_{2,n}$ admits a K\"ahler--Einstein metric.
\end{maincorollary}

\noindent\textbf{Acknowledgments.} The author would like to thank her PhD supervisor, Hamid Abban, for suggesting this problem and for his constant support. She would also like to thank Kento Fujita for many valuable discussions and for carefully reading earlier drafts of this paper. The research is supported by the University of Nottingham EPSRC DTP grant EP/W524402/1. Partial support was also received from the Royal Society International Collaboration Award ICA$\backslash$1$\backslash$231019.

\section{Preliminaries}\label{preliminary}
In this section we recall the notions needed to define K-stability, and then state and prove a simple version of the Abban--Zhuang method. Readers familiar with this may proceed directly to the next section.

\subsection{Basic Notions}\label{basic notions}
We work over $\mathbb{C}$, and all varieties are assumed to be integral, normal and projective, unless otherwise stated. A valuation on a variety $X$ is a valuation $v\colon \bb{C}(X)^{\times}\rightarrow \bb{R}$ which is trivial on $\bb{C}^{\times}$. It is a trivial valuation if it is identically zero. 
    
A ($\mathbb{Q}$-)Fano variety $X$ is a variety with at worst klt singularities such that the anti-canonical divisor~$-K_X$ is ($\mathbb{Q}$-)Cartier and ample. A log pair $(X,\Delta)$ consists of a variety $X$ together with an effective $\mathbb{Q}$-divisor $\Delta$ such that $K_X+\Delta$ is a $\Q$-Cartier $\Q$-divisor. It is called a klt pair if $(X,\Delta)$ is moreover klt. And a log Fano pair is a klt pair such that $-(K_X+\Delta)$ is ample.

For a log pair $(X, \Delta)$, the non-lc locus is $\operatorname{Nlc}(X, \Delta):=\{x\in X | (X, \Delta)$ is not  log canonical at $x\}$. Moreover, given a closed subset $Z\subset X$ and an effective divisor $D$ on $X$, the log canonical threshold (lct) is defined as $\operatorname{lct}_Z(X, \Delta; D):=\sup\{\lambda\in \bb{Q}_{\geq 0}\mid Z\not\subseteq\operatorname{Nlc}(X, \Delta+\lambda D)  \}$.

\begin{definition}
    (Filtration) Let $V$ be a finite-dimensional vector space. A (non-increasing, bounded) \textit{filtration} $\mathcal{F}$ on $V$ is given by a collection of subspaces $\mathcal{F}^{\lambda}V$ indexed by a totally ordered abelian monoid~$\Lambda$ such that $\mathcal{F}^{\lambda_0}V=V, \mathcal{F}^{\lambda_1}V=0$ for some $\lambda_0, \lambda_1\in \Lambda$ and $\mathcal{F}^{\lambda}V\subseteq \mathcal{F}^{\lambda^{\prime}}V$ whenever $\lambda\geq \lambda^{\prime}$. 
\end{definition}

\begin{definition}\label{m-basia type}($m$-basis type divisor)
    Let $X$ be a variety and $L$ be a $\Q$-Cartier $\Q$-divisor on $X$. Let~$m$ be a positive integer such that $V=H^0(X,mL)$ is non-zero. Then an \textit{$m$-basis type divisor} $D_m$ of $L$ is of the form $D_m=\sum \operatorname{div}(s_i)$ (by abuse of notation, here we use $\operatorname{div}(s)$ to represent the effective divisor~$\operatorname{div}(s)+mL$) where $\{s_1,\dots,s_{h^0(X,mL)}\}$ form a basis of $H^0(X,mL)$. And an \textit{$m$-basis type $\Q$-divisor} $D$ of $L$ is $D=\frac{1}{m\cdot h^0(X,mL)}D_m$ for some $m$-basis type divisor $D_m$. It is obvious that $D\sim_{\Q}L$. If $\mathcal{F}$ is a filtration on $V$, we say $D$ is an $m$-basis type ($\Q$-)divisor of $L$ \textit{compatible with $\mathcal{F}$} if for every $\lambda\in \Lambda$, the subspace $\mathcal{F}^{\lambda}V\subseteq V$ is spanned by a subset of $\{s_1,\dots, s_{h^0(X,mL)}\}$.
\end{definition}

\begin{definition}\label{S_m}
    In the setting of the previous definition, let $v$ be a valuation on $X$. For a Cartier divisor~$D$ on $X$, locally defined by $f=0$ on some open affine subset of $X$, we view $f$ as a rational function on $X$ and set $v(D):=v(f)$. For a $\Q$-Cartier $\bb{Q}$-divisor $D$ on X we define $v(D):=\frac{1}{m}v(mD)$ for any $m\in \bb{Z}_{>0}$ such that $mD$ is Cartier. It is easy to see from the definition that $v$ is linear, i.e.\ for a $\bb{Q}$-Cartier $\bb{Q}$-divisor~$D=\sum a_iD_i$ we have $v(D)=\sum a_i v(D_i)$. 
        
    Then we define $$S_m(L;v)=\sup v(D),$$
    where the supremum runs over all $m$-basis type $\Q$-divisors of $L$.
\end{definition}

\begin{definition}\label{volume and S-invariant}
    Let $X$ be a variety of dimension $n$, and $L$ be a big $\Q$-Cartier $\Q$-divisor on $X$. The \textit{volume} of $L$ is defined as 
    $$\mathrm{vol}(L)=\lim\limits_{m\rightarrow \infty}\frac{\dim H^0(X,mL)}{m^n/n!},$$ for sufficiently divisible integers $m$. 

    We call $E$ a prime divisor \textit{over} $X$ if there exists a birational morphism $\pi:Y\rightarrow X$ such that $E$ is a prime divisor on $Y$. The image $\pi(E)$ is called the \textit{center} of $E$, denoted by $C_X(E)$. 
        
    Given a prime divisor $E$ over $X$, the \textit{pseudo-effective threshold} is defined as 
    $$T(L;E)=\sup \{t\in \Q_{>0}\mid \mathrm{vol}(\pi^*{L}-tE)> 0\}.$$
    Note that in the vector space $N^1(Y)_{\bb{R}}$, the closure of the cone of big divisors coincides with the closure of the cone of effective divisors: $\overline{\operatorname{Big}(Y)}=\overline{\operatorname{Eff}(Y)}$. Therefore $T(L;E)=\sup \{t\in \Q_{>0}\mid [\pi^*{L}-tE]\in \operatorname{Eff}(Y)  \}$. This justifies the name of pseudo-effective threshold.
        
    The \textit{movable threshold} is defined as 
    \begin{align*}
        \eta(L;E)=\sup\{&t\in\Q_{>0}\mid \exists\ m\in \bb{Z}_{>0}\text{ such that}\\
        &\text{any codimension one closed subset of base locus of } |m\pi^*L-mtE| \text{ is exceptional over } X\}.
    \end{align*}
    Clearly we have $\eta(L;E)\leq T(L;E)$.
         
    We define $$S(L;E)=\frac{1}{\mathrm{vol}(L)}\int^{T(L;E)}_0 \mathrm{vol}(\pi^*L-tE)dt.$$

    It is easy to show that the values $T(L;E), \eta(L;E)$ and $S(L;E)$ only depend on the valuation $\operatorname{ord}_E$ and is independent of the choice of the birational model $\pi:Y\rightarrow X$. 
\end{definition}
\begin{remark}
    By \autocite[Theorem 3.3]{BJThresholdsValuationandK-stability}, when $v=\mathrm{ord}_E$ for some prime divisor $E$ over $X$, we have~$\lim\limits_{m\rightarrow\infty}S_m(L; \mathrm{ord}_E)=S(L;E).$
\end{remark}

\begin{lemma}\label{lemma on movable threshold}
    Notation as above. Assume $n=\dim X\geq 2$. Take any smooth closed point $x\in X$, and let $E$ be the exceptional divisor of the blowup at $x$.

    1. We have $T(L;E)=\sup \{\mathrm{mult}_x(D)\mid 0\leq D\sim_{\Q} L\}$.

    2. If $L$ is very ample, then $\eta(L;E)\leq \sqrt{(L^n)}$.
\end{lemma}
\begin{proof}
    Note that equivalently we have $T(L;E)=\sup\{t\in \bb{Q}_{>0}\ |\ \pi^*D-tE$ is effective for some $D\sim_{\bb{Q}}L\}$, then the first claim is straightforward. 

    For the second claim, assuming $\eta(L;E)> \sqrt{(L^n)}$, we can find effective $\Q$-Cartier $\Q$-divisors $D_1,D_2\sim_{\Q}L$ with no common components and $\mathrm{mult}_x(D_i)> \sqrt{(L^n)}$ for $i=1,2$. Pick $H_3,\dots ,H_n\in |L|$ in general position passing through $x$, such that $\operatorname{mult}_x(H_j)=1$, then we have $$(L^n)=(D_1\cdot D_2 \cdot H_3\dots H_n)\geq \mathrm{mult}_x(D_1)\cdot \mathrm{mult}_x(D_2)\cdot \prod_{j=3}^n \operatorname{mult}_x(H_j)> (L^n).$$
    This gives a contradiction, hence we complete the proof.
\end{proof}
   
\begin{definition}($\beta$-invariant)
    Let $(X, \Delta)$ be a log Fano pair and $E$ be a prime divisor over $X$.

    The \textit{log discrepancy} of $E$ is defined as 
    $$A_{X,\Delta}(E)=1+\mathrm{ord}_{E}(K_Y-\pi^*(K_X+\Delta)),$$
    where $\pi: Y\rightarrow X$ is a birational morphism such that $E$ is a divisor on $Y$.
       
    Let $v$ be a non-trivial valuation on $X$, we can also define $A_{X,\Delta}(v)$ the log discrepancy of $v$. For the rigorous definition, we refer to \cite{JonssonMustata}. This definition is compatible with the one for divisors: if $E$ is a prime divisor over $X$, it induces a valuation $\operatorname{ord}_E$, and $A_{X,\Delta}(E)=A_{X,\Delta}(\operatorname{ord}_E)$.

    The \textit{$\beta$-invariant} of $E$ is defined as 
    $$\beta_{X, \Delta}(E)=A_{X, \Delta}(E)-S(-K_X-\Delta;E).$$
\end{definition}

\begin{definition}\label{delta invariant}($\delta$-invariant)
    Let $(X, \Delta)$ be a klt pair. Let $Z\subset X$ be a (possibly reducible) closed subset and $L$ be an ample $\bb{Q}$-Cartier $\Q$-divisor on $X$. Define $$\delta_m(X, \Delta;L)=\inf\limits_{D} \operatorname{lct}(X, \Delta;D),\quad\delta_{Z, m}(X, \Delta;L)=\inf\limits_{D} \operatorname{lct}_Z(X, \Delta;D),$$ where the infimum runs over all $m$-basis type $\Q$-divisor $D$ of $L$.
    The \textit{stability threshold}, also known as the \textit{$\delta$-invariant} is defined as 
    $$\delta(X, \Delta;L)=\limsup_{m\rightarrow \infty}\delta_m(X, \Delta;L),\quad\delta_Z(X, \Delta;L)=\limsup_{m\rightarrow \infty}\delta_{Z, m}(X, \Delta;L).$$
    When $(X, \Delta)$ is a log Fano pair, we sometimes write $\delta_{Z}(X, \Delta):=\delta_{Z}(X, \Delta; -K_X-\Delta)$ for simplicity, and we may omit $\Delta$ if $\Delta=0$.
\end{definition}

Equivalently, it can be defined in the following way.
\begin{theorem}(\autocite[Theorem 4.4]{BJThresholdsValuationandK-stability})\label{delta invariant as inf A/E}
    With the same setting as above, if the closed subset $Z\subset X$ is moreover irreducible, then we have
    \begin{align*}
        \delta(X, \Delta;L)
        =\inf \frac{A_{X,\Delta}(E)}{S(L;E)},\quad \delta_Z(X, \Delta;L)
        =\inf \frac{A_{X, \Delta}(E)}{S(L;E)}, 
    \end{align*}
    where the infimum runs over all prime divisors $E$ over $X$(resp. all such divisors $E$ with $C_X(E)\supseteq Z$).
\end{theorem}

The original definition of K-stability is not used in this paper, hence we will take the equivalent condition proved in \cite{FujitaOdakaLowerAlpha}, \cite{BJThresholdsValuationandK-stability} as the definition.
\begin{definition}(K-stability)
    A log Fano pair $(X, \Delta)$ is K-stable (respectively, K-semistable) if~$\delta(X, \Delta)>1$ (respectively, $\geq 1$).
\end{definition}
Together with \autocite[Theorem 1.2]{LXZ}, we have the following equivalent characterization:
\begin{proposition}
    Let $(X, \Delta)$ be a log Fano pair. Then it is K-stable (respectively, K-semistable) if and only if $ \delta_x(X, \Delta)>1$ (respectively, $\geq 1$) for all closed points $x\in X$.
\end{proposition}

We remark that K-polystability is weaker than K-stability while stronger than K-semistability, which does not admit a direct valuative characterization. We state the following theorem for later use.
\begin{theorem}\label{finite automorphism}
    1.(\cite{ReductiveAutomorphism}) If $(X, \Delta)$ is K-polystable, then $\mathrm{Aut}(X, \Delta)$ is reductive.

    2.(\cite{FiniteAutomorphism}) If $(X, \Delta)$ is K-stable, then $|\mathrm{Aut}(X, \Delta)|<\infty $.
        
    3. If $\mathrm{Aut}(X, \Delta)$ is finite, then $(X, \Delta)$ is K-stable if and only if it is K-polystable.
\end{theorem}

\subsection{The Abban--Zhuang method}\label{abban-zhuang method}
In this subsection, we present a strong tool for estimating a lower bound of the local stability threshold. The theory is well-built in \cite{AZ1}. For the paper to be self-contained, in Theorem \ref{AZ Admissible Flag} we state and prove a simplified version that is sufficient for the proof of Theorem A. Its formulation allows one to consider a broader range of filtrations that may be of independent interest. 
   
\begin{definition}\label{Multigraded Linear Series}(Multigraded linear series)
    Let $X$ be an $n$-dimensional variety. Let $L_1, \dots, L_r$ be $\bb{Q}$-Cartier $\Q$-divisors on $X$. An \textit{$\bb{N}^r$-graded linear series $W_{\vec{\bullet}}$ on $X$ associated to $L_i$'s} consists of finite-dimensional subspaces $$W_{\vec{a}}\subseteq H^0(X,\mathcal{O}_X(\vec{a}\cdot \vec{L}))$$ for each $\vec{a}\in \bb{N}^r$ (where $\vec{a}\cdot \vec{L}=a_1L_1+\dots +a_rL_r$) such that $W_{\vec{0}}=\bb{C}$ and $W_{\vec{a}_1}\cdot W_{\vec{a}_2}\subseteq W_{\vec{a}_1+\vec{a}_2}$ for all~$\vec{a}_1, \vec{a}_2\in \bb{N}^r$.

    The support $\mathrm{Supp}(W_{\vec{\bullet}})\subseteq \bb{R}^r$ of $W_{\vec{\bullet}}$ is defined as the closed convex cone spanned by $\{\vec{a}\in \bb{N}^r\mid W_{\vec{a}}\neq 0\}.$ We say $W_{\vec{\bullet}}$ \textit{has bounded support} if the intersection $\Delta_{\text{Supp}}=\Delta_{\text{Supp}}(W_{\vec{\bullet}}):=\mathrm{Supp}(W_{\vec{\bullet}})\cap (\{1\}\times \bb{R}^{r-1})$ is a bounded set in $\bb{R}^{r-1}$.

    When $W_{\vec{\bullet}}$ has bounded support, we set $h^0(W_{m,\vec{\bullet}}):=\sum_{\vec{a}\in\bb{N}^{r-1}}\dim W_{m,\vec{a}}$ for all $m\in \bb{N}$, and define the \textit{volume} of $W_{\vec{\bullet}}$ as $$\operatorname{vol}(W_{\vec{\bullet}}):=\limsup_{m\rightarrow\infty}\frac{h^0(W_{m,\vec{\bullet}})}{m^{n+r-1}/(n+r-1)!}.$$

    We say that $W_{\vec{\bullet}}$ \textit{contains an ample series} if:

    1. The interior of $\mathrm{Supp}(W_{\vec{\bullet}})$ is non-empty;

    2. for any $\vec{a}\in \mathrm{int}(\mathrm{Supp}(W_{\vec{\bullet}}))\cap \bb{N}^r, W_{k\vec{a}}\neq 0$ for all $k\gg 0$;

    3. there exists some $\vec{a}_0\in \mathrm{int}(\mathrm{Supp}(W_{\vec{\bullet}}))\cap \bb{N}^r$ and a decomposition $\vec{a}_0\cdot \vec{L}= A + E$ with $A$ an ample $\bb{Q}$-Cartier $\Q$-divisor and $E$ an effective $\Q$-divisor such that $H^0(X,mA)\subseteq W_{m\vec{a}_0}$ for all sufficiently divisible $m$.
\end{definition}
    
\begin{definition}\label{filtraiton on multigraded linear series}\autocite[Definition 2.10]{AZ1}
    Let $W_{\vec{\bullet}}$ be an $\bb{N}^r$-graded linear series. A filtration $\mathcal{F}$ on $W_{\vec{\bullet}}$ (indexed by a totally ordered abelian monoid  $\Lambda$) is given by a filtration on each $W_{\vec{a}}$ such that $$\mathcal{F}^{\lambda_1}W_{\vec{a}_1} \cdot \mathcal{F}^{\lambda_2}W_{\vec{a}_2}\subseteq \mathcal{F}^{\lambda_1+\lambda_2}W_{\vec{a}_1+\vec{a}_2}$$ for all $\lambda_i\in \Lambda$ and all $\vec{a}_i\in \bb{N}^r$.
\end{definition}

\begin{definition}\label{basis type divisor}\autocite[Definition 2.11]{AZ1}
    Let $W_{\vec{\bullet}}$ be an $\bb{N}^r$-graded linear series with bounded support. For any $m\in \bb{N}$, let $N_m=h^0(W_{m,\vec{\bullet}})$. We say that $D$ is an \textit{$m$-basis type divisor (resp. $\Q$-divisor) of $W_{\vec{\bullet}}$} if~$N_m\neq 0$ and there exist basis type divisors $D_{\vec{a}}$ of $W_{m,\vec{a}}$ for each $\vec{a}\in \bb{N}^{r-1}$ with $W_{m,\vec{a}}\neq 0$ such that 
    $$D=\sum\limits_{\vec{a}\in\bb{N}^{r-1}} D_{\vec{a}} \quad\left(\text{resp. }D=\frac{1}{mN_m}\sum\limits_{\vec{a}\in\bb{N}^{r-1}} D_{\vec{a}} \right).$$
    Let $\mathcal{F}$ be a filtration on $W_{\vec{\bullet}}$, we say $D$ is \textit{compatible} with $\mathcal{F}$ if all $D_{\vec{a}}$ are compatible with $\mathcal{F}$.
\end{definition}

\begin{example}\label{refinement}
    Let $X$ be a variety and $L_1, \dots, L_r$ be $\Q$-Cartier $\Q$-divisors on $X$. Let $V_{\vec{\bullet}}$ be an $\bb{N}^{r}$-graded linear series associated to $\vec{L}=(L_1, \dots, L_r)$. Let $Y\subset X$ be a prime Cartier divisor, then $Y|_Y\in \operatorname{CaCl}(Y)$ is well-defined. We fix a representative of it and also denote it as $Y|_Y$. 
         
    Under the canonical embedding $H^0(X,\vec{a}\cdot \vec{L}-jY)\xhookrightarrow{} H^0(X,\vec{a}\cdot \vec{L})$, we set $$V_{\vec{a}}(-jY):=V_{\vec{a}}\cap H^0(X,\vec{a}\cdot \vec{L}-jY)\subset H^0(X,\vec{a}\cdot \vec{L}).$$ 
    Then we can consider the $\bb{N}^{r+1}$-graded linear series $W_{\vec{\bullet}}$ on $Y$ given by $$W_{\vec{a},j}=\operatorname{Im}\left(V_{\vec{a}}\left(-jY\right)\rightarrow  H^0(X,\vec{a}\cdot \vec{L}-jY)\rightarrow H^0(Y, (\vec{a}\cdot \vec{L}-jY)\vert_Y )\right).$$ We call $W_{\vec{\bullet}}$ the \textit{refinement} of $V_{\vec{\bullet}}$ by $Y$. It is a multi-graded linear series associated to $L_1|_Y, \dots, L_r|_Y$ and~$-Y|_Y$.

    On the other hand, the divisor $Y$ induces an $\bb{R}$-filtration $\mathcal{F}_Y$ on $V_{\vec{\bullet}}$, given by $$\mathcal{F}_Y^{\lambda}V_{\vec{a}}:=\{s\in V_{\vec{a}}\mid \operatorname{ord}_Y(s)\geq \lambda\}.$$ We call $\mathcal{F}_Y$ the filtration on $V_{\vec{\bullet}}$ induced by $Y$. Note for any $j\in\bb{N}$, we have $W_{\vec{a},j}=\mathcal{F}_Y^jV_{\vec{a}}/ \mathcal{F}_Y^{j+1}V_{\vec{a}}$.

    Let $$Y_{\bullet}: X=Y_0\supset Y_1\supset \dots \supset Y_{\ell}$$ be an \textit{admissible flag} of length $\ell$ on $X$, i.e.\ each $Y_i$ is an irreducible codimension one subvariety in $Y_{i-1}$ that is smooth at the generic point of $Y_{\ell}$. If moreover each $Y_{i+1}$ is a Cartier divisor on $Y_i$, we can inductively construct a $\bb{N}^{r+\ell}$-graded linear series $W_{\vec{\bullet}}^{(\ell)}$ on $Y_{\ell}$.  We call it the \textit{refinement} of $V_{\bullet}$ by $Y_{\bullet}$. Note that all~$W_{\vec{\bullet}}^{(\ell)}$'s have bounded support and contains an ample series if $V_{\vec{\bullet}}$ does, see \autocite[Example 2.6]{AZ1}.

    Similarly, the admissible flag $Y_{\bullet}$ induces a $\bb{R}^{\ell}$-filtration $\mathcal{F}_{Y_{\bullet}}$ on $V_{\vec{\bullet}}$, characterized by 
    $$W^{(\ell)}_{\vec{a},\vec{b}}=\mathcal{F}_{Y_{\bullet}}^{\vec{b}}V_{\vec{a}}/ \bigcup_{\vec{\lambda}>\vec{b}}\mathcal{F}_{Y_{\bullet}}^{\vec{\lambda}}V_{\vec{a}}.$$
    Here we use the lexicographic order. We call $\mathcal{F}_{Y_{\bullet}}$ the filtration on $V_{\vec{\bullet}}$ induced by $Y_{\vec{\bullet}}$.

    We have the monoid  $$\Gamma(W_{\vec{\bullet}})=\Gamma_{Y_{\bullet}}(W_{\vec{\bullet}}):=\{\vec{a}\mid W_{\vec{a}}\neq 0\}\subset \bb{N}^{r+\ell}.$$
    Recall that for every monoid $\Gamma\subset \bb{N}^{N+1}$, we can associate it to a convex body $\Delta(\Gamma)\subset \bb{R}^N$ by setting 
    $$\begin{array}{c}
        \Sigma(\Gamma)=\text{the\ smallest\ closed\ convex\ cone\ containing } \Gamma\subset \bb{R}^{N+1},\\
        \text{and\ }\Delta(\Gamma)=\Sigma(\Gamma) \cap (\{1\}\times \bb{R}^N).
    \end{array}$$ 
    Let $\Gamma_m=\Gamma(W_{\vec{\bullet}})\cap \left( \{m\}\times \bb{N}^{r+\ell-1}  \right)$. When $W_{\vec{\bullet}}$ contains an ample series and has bounded support, the atomic measure $$\rho_m=\frac{1}{m^{r-1+\ell}}\sum_{\vec{a}\in \Gamma_m}\delta_{m^{-1}\vec{a}}$$ on $\Delta(W_{\vec{\bullet}})$ converges weakly to the Lebesgue measure on it as $m\rightarrow\infty$, by \autocite[Th\'eor\`eme 1.12]{boucksom2014corps}.

    If moreover the length of the flag is $n=\dim X$, the associated $\Delta(W_{\vec{\bullet}}):=\Delta(\Gamma(W_{\vec{\bullet}}))\subset \bb{R}^{r+n-1}$ is called the \textit{Okounkov body} of $W_{\vec{\bullet}}$ (see \cite{lazarsfeld2009convex} for details). We have $\operatorname{vol}(W_{\vec{\bullet}})=(n+r-1)!\cdot \operatorname{vol}(\Delta(W_{\vec{\bullet}}))$. 
\end{example}
    
Now, we generalize the notion of $S$-invariant and $\delta$-invariant (see Definitions \ref{S_m}, \ref{volume and S-invariant}, and \ref{delta invariant}) and Theorem \ref{delta invariant as inf A/E} to the multi-graded linear series setting.
\begin{definition}\autocite[Definition 2.12]{AZ1}
    Let $(X, \Delta)$ be a klt pair, and let $Z$ be a closed subset of $X$. Let $W_{\vec{\bullet}}$ be an $\bb{N}^{r}$-graded linear series on $X$ with bounded support, let $\mathcal{F}$ be a filtration on $W_{\vec{\bullet}}$, and let $v$ be a non-trivial valuation on $X$ with $A_{X, \Delta}(v)<\infty$. We define 
    $$S_m(W_{\vec{\bullet}},\mathcal{F};v):=\sup\limits_{D}v(D),$$ and $$S(W_{\vec{\bullet}},\mathcal{F};v)=\limsup\limits_{m\rightarrow\infty}S_m(W_{\vec{\bullet}},\mathcal{F};v),$$
    where the supremum runs over all $m$-basis type divisor $\Q$-divisor $D$ of $W_{\vec{\bullet}}$ compatible with $\mathcal{F}$. When the valuation $v=\mathrm{ord}_E$ is induced by a divisor, we write $S_m(W_{\vec{\bullet}},\mathcal{F};E)=S_m(W_{\vec{\bullet}},\mathcal{F};\mathrm{ord}_E)$. We omit $\mathcal{F}$ in the notation when it is the trivial filtration (i.e. $\mathcal{F}^{\lambda}(W_{\vec{a}})=W_{\vec{a}}$ when $\lambda\leq 0$ and $\mathcal{F}^{\lambda}(W_{\vec{a}})=0$ when $\lambda>0$). We also set $$\delta_m(W_{\vec{\bullet}},\mathcal{F})=\inf\limits_{D}\mathrm{lct}(X, \Delta;D);\quad\delta_{Z,m}(W_{\vec{\bullet}},\mathcal{F})=\inf\limits_{D}\mathrm{lct}_Z(X, \Delta;D),$$ where the infimum runs over all $m$-basis type $\Q$-divisors of $W_{\vec{\bullet}}$ that are compatible with $\mathcal{F}$. Similarly we define $$\delta(W_{\vec{\bullet}},\mathcal{F})=\limsup\limits_{m\rightarrow\infty}\delta_m(W_{\vec{\bullet}},\mathcal{F});\quad\delta_Z(W_{\vec{\bullet}},\mathcal{F})=\limsup\limits_{m\rightarrow\infty}\delta_{Z,m}(W_{\vec{\bullet}},\mathcal{F}).$$ 
\end{definition}

\begin{remark}
    Let $L$ be a big line bundle on $X$ and $V_{\bullet}$ be the complete linear series associated to $L$, then the notions are compatible with those in subsection \ref{basic notions}. Explicitly, we have $\operatorname{vol}(V_{\bullet})=\operatorname{vol}(L), S(V_{\bullet};v)=S(L;v),  \delta(V_{\bullet})=\delta(L),$ and $\delta_Z(V_{\bullet})=\delta_Z(L)$ for any valuation $v$ on $X$ and any closed subset $Z\subset X$.
\end{remark}

\begin{lemma}\label{valuative criterion for multigraded} \autocite[Lemma 2.9, Proposition 3.1]{AZ1}
    With the same setting as previous definition, if $W_{\vec{\bullet}}$ contains an ample series, then $$S(W_{\vec{\bullet}},\mathcal{F};v)=\lim\limits_{m\rightarrow\infty}S_m(W_{\vec{\bullet}},\mathcal{F};v),$$ for any valuation $v$ with linear growth (i.e. there exists an ample line bundle $L$ and a constant $T>0$ such that $v(D)<mT$ for any $D\in |mL|$).
        
    If moreover $(X, \Delta)$ is a klt pair and the closed subset $Z\subset X$ is irreducible, then for any filtration $\mathcal{F}$ on $W_{\vec{\bullet}}$ we have 
    \begin{align*}
        \delta(W_{\vec{\bullet}})&=\delta(W_{\vec{\bullet}}, \mathcal{F})=\inf\limits_{E}\frac{A_{X, \Delta}(E)}{S(W_{\vec{\bullet}};E)}=\inf\limits_{v}\frac{A_{X, \Delta}(v)}{S(W_{\vec{\bullet}};v)},\\
        \delta_Z(W_{\vec{\bullet}})&=\delta_Z(W_{\vec{\bullet}}, \mathcal{F})=\inf\limits_{E,C_X(E)\supseteq Z}\frac{A_{X, \Delta}(E)}{S(W_{\vec{\bullet}};E)}=\inf\limits_{v, C_X(v)\supseteq Z}\frac{A_{X, \Delta}(v)}{S(W_{\vec{\bullet}};v)},
    \end{align*}
    where the first infimum runs over all divisors $E$ over $X$ (resp. all divisors $E$ over $X$ whose center contains~$Z$) and the second infimum runs over all non-trivial valuations (resp. non-trivial valuations whose center contains $Z$) with $A_{X, \Delta}(v)<\infty$.
\end{lemma}

\begin{remark}\label{volume function and compute S}
    Given a multigraded linear series $W_{\vec{\bullet}}$ with bounded support $\Delta_{\text{Supp}}:=\Delta_{\text{Supp}}(W_{\vec{\bullet}})$, there is a continuous function $$\operatorname{vol}_{W_{\vec{\bullet}}}\colon \operatorname{Int}\left(\Delta_{\text{Supp}}\right)\rightarrow \bb{R}_{\geq{0}}$$ such that for any $\gamma\in \operatorname{Int}\left(\Delta_{\text{Supp}}\right)\cap\bb{Q}^{r-1}$, the value of $\operatorname{vol}_{W_{\vec{\bullet}}}(\gamma)$ equals the volume of the singly-graded linear series 
    $\bigoplus_m W_{m,m\gamma}$. Then $\operatorname{vol}_{W_{\vec{\bullet}}}$ can be extended continuously to $\Delta_{\text{Supp}}$. 

    If there is a filtration $\mathcal{F}$ on $W_{\vec{\bullet}}$, we can define the multigraded linear series $W_{\vec{\bullet}}^t$ by $W_{m,\vec{a}}^t=\mathcal{F}^{mt}W_{m,\vec{a}}$. Then, according to \autocite[Corollary 2.13, Lemma 2.10]{AZ1}, we have $$S(W_{\vec{\bullet}};\mathcal{F})=\frac{1}{\operatorname{vol}\left(W_{\vec{\bullet}}\right)}\int_0^{\infty} \operatorname{vol}(W_{\vec{\bullet}}^t)dt=\frac{1}{n!\operatorname{vol}\left(\Delta\left(W_{\vec{\bullet}}\right)\right)}\int_0^{\infty}\int_{\Delta_{\text{Supp}}^t}\operatorname{vol}_{W_{\vec{\bullet}}^t}(\gamma)d\gamma dt.$$
\end{remark}

\begin{lemma-definition}\autocite[Lemma-Definition 2.14]{AZ1}\label{c_1 and fixed part}
    With the same notation as Definitions \ref{Multigraded Linear Series} and \ref{basis type divisor}, we see that the class of each $m$-basis type divisor of $W_{\vec{\bullet}}$ is independent of the choice of basis. Denote it by $c_1(W_{m,\vec{\bullet}})\in \operatorname{Pic}(X)_{\bb{R}}$. If $W_{\vec{\bullet}}$ contains an ample series, then the limit $$c_1(W_{\vec{\bullet}}):=\lim\limits_{m\rightarrow\infty}\frac{c_1(W_{m,\vec{\bullet}})}{mh^0(W_{m,\vec{\bullet}})}
    $$ exists in $\mathrm{Pic}(X)_{\bb{R}}$. 

    Let $F_{\vec{a}}$ (resp. $M_{\vec{a}}$) be the fixed part (resp. movable part) of the linear system $W_{\vec{a}}$ for every $\vec{a}\in \bb{N}^{r}$ such that $W_{\vec{a}}\neq 0$. For each $m$, let $$F_m(W_{\vec{\bullet}})=\frac{1}{mN_m}\sum\limits_{\vec{a}\in\bb{N}^{r-1}}\dim (W_{m,\vec{a}})\cdot F_{m,\vec{a}}.$$ Then $\lim\limits_{m\rightarrow\infty}\mathrm{ord}_D F_m(W_{\vec{\bullet}})$ exists for every prime divisor $D\subset X$.
\end{lemma-definition}

Note that although $\dim_{\bb{R}}\operatorname{Pic}(X)_{\bb{R}}$ is in general infinite, the class $c_1(W_{m,\vec{\bullet}})\in \operatorname{Pic}(X)_{\bb{R}}$ always lies in a finite dimensional vector space spanned by $\{[L_1], \dots, [L_{r}]\}$ for all $m\in \bb{N}$. Therefore the first limit in Definition \ref{c_1 and fixed part} is the limit in a finite-dimensional space $\langle[L_1], \dots , [L_{r}]\rangle\subseteq \operatorname{Pic}(X)_{\bb{R}}$.

\begin{definition}\label{almost complete}\autocite[Definition 2.16]{AZ1}
    Let $L$ be a big line bundle on a variety $X$, and $W_{\vec{\bullet}}$ be an~$\bb{N}^{r}$-graded linear series. We say $W_{\vec{\bullet}}$ is \textit{almost complete} (with respect to $L$) if:

    (1) there are at most finitely many prime divisors $D\subset X$ with $\lim\limits_{m\rightarrow\infty}\mathrm{ord}_D F_m(W_{\vec{\bullet}})>0$. Then we denote $$F(W_{\vec{\bullet}})=\sum\limits_{D} \left(\lim_{m\rightarrow\infty}\mathrm{ord}_D F_m\left(W_{\vec{\bullet}}\right)\right)\cdot D.$$

    (2) for every $\gamma\in\operatorname{Int}(\Delta_{\text{Supp}})\cap \Q^{r-1}$ and all sufficiently divisible integers $m$, we have $\rvert M_{m,m\gamma}\lvert \subseteq \rvert L_{m,\gamma}\rvert$ for some $L_{m,\gamma}\equiv \ell_{m,\gamma}L$ and some $\ell_{m,\gamma}\in\bb{N}$ such that $$\frac{\dim(W_{m,m\gamma})}{h^0(X,\ell_{m,\gamma}L)}=\frac{\dim(M_{m,m\gamma})}{h^0(X,\ell_{m,\gamma}L)}\rightarrow 1$$ as $m\rightarrow\infty$.
\end{definition}

\begin{example}\label{ample is almost complete}\autocite[Example 2.12]{AZ1}
    Let $L$ be an ample line bundle on $X$, and $H\in |L|$ be an integral divisor. Let $V_{\bullet}$ be the complete linear system of $rL$ for some $r\in \bb{Z}_{>0}$, and $W_{\bullet, \bullet}$ be the refinement of~$V_{\bullet}$ by $H$, then the $\bb{N}^2$-graded linear series $W_{\bullet, \bullet}$ is almost complete and $F(W_{\bullet, \bullet})=0$. More generally, if~$Y_{\bullet}:Y_0\supset Y_1\supset \dots \supset Y_{\ell}$ is an admissible flag on $X$ such that each $Y_i$ is Cartier on $Y_{i-1}$ and $Y_i\sim m_iL|_{Y_{i-1}}$ for some $m_i\in\bb{Z}_{>0}$, then the $\bb{N}^{\ell+1}$-graded linear series refined from $V_{\bullet}$ is almost complete and $F(W_{\vec{\bullet}})=0$.
\end{example}
    
\begin{lemma}\label{induction formula}\autocite[Lemma 2.13]{AZ1}
    Let $L$ be a big line bundle on $X$ and $W_{\vec{\bullet}}$ be an $\bb{N}^{r+1}$-graded linear series. Assume $W_{\vec{\bullet}}$ has bounded support, contains an ample series and is almost complete with respect to~$L$. Then 

    1. $F(W_{\vec{\bullet}})$ is $\bb{R}$-Cartier;

    2. there exists a constant $\mu=\mu(X,L,W_{\vec{\bullet}})$ such that $c_1(M_{\vec{\bullet}})=\mu L$ in $N^1(X)_{\bb{R}}$ and $$S(W_{\vec{\bullet}};v)=\mu\cdot S(L;v)+v(F(W_{\vec{\bullet}}))$$ for all non-trivial valuations of linear growth.
\end{lemma}

Now we are ready to state the main tool that we will use later to prove Theorem \ref{main theorem}.

Let $X$ be a variety with at worst klt singularities, and $V_{\vec{\bullet}}$ an $\bb{N}^{r}$-graded linear series on $X$ containing an ample series and with bounded support. Let $Y\subset X$ be a klt prime Cartier divisor on $X$, and $W_{\vec{\bullet}}$ be the refinement of $V_{\vec{\bullet}}$ by $Y$.

Let $\mathcal{F}$ be an arbitrary filtration on $W_{\vec{\bullet}}$, indexed by some totally ordered abelian monoid $\Lambda$. We can lift it to a filtration $\mathcal{F}_1$ on $V_{\vec{\bullet}}$ indexed by $\bb{R}\oplus \Lambda$ with lexicographic order, defined by 
\begin{equation}\label{eq: lift of filtration}
    \mathcal{F}_1^{(t,\lambda)} V_{\vec{a}}=\{s\in V_{\vec{a}}\mid \text{either\ }\operatorname{ord}_{Y}(s)>t, \text{\ or\ } \operatorname{ord}_{Y}(s)=t \text{\ and\ } (s-tY)\vert_{Y}\in\mathcal{F}^{\lambda} W_{\vec{a},t} \},
\end{equation}
for any $t\in \bb{R}, \lambda\in\Lambda$, and any $\vec{a}\in\bb{N}^{r}$ such that $V_{\vec{a}}\neq 0$.

It is worth noting that here, and in the following theorem, we consider arbitrary filtrations, this generalizes the setting in the original \cite{AZ1}, which only considers filtrations induced by admissible flags.
\begin{lemma}\label{lemma: induction step of AZ}
    Let $Z\subset X$ be a (possibly reducible) closed subset of $X$ that intersects $Y$, we have 
    \begin{equation}\label{eq: induction step of AZ}
        \delta_Z(X;V_{\vec{\bullet}}, \mathcal{F}_1)\geq \min \left\{\frac{1}{S(V_{\vec{\bullet}};Y)},\ \delta_{Z\cap Y}(Y;W_{\vec{\bullet}},\mathcal{F}) \right\}.
    \end{equation}
\end{lemma}
\begin{proof}
    Let $D$ be an $m$-basis type $\bb{Q}$-divisor of $V_{\vec{\bullet}}$ compatible with $\mathcal{F}_{1}$, we have $D=S_m(V_{\vec{\bullet}};Y)\cdot Y+\Gamma$ such that $\operatorname{Supp}(\Gamma)$ does not contain $Y$ and $\Gamma\vert_{Y}$ is an $m$-basis type $\bb{Q}$-divisor of $W_{\vec{\bullet}}$ compatible with $\mathcal{F}$. Denote~$\lambda_m:=\min \left\{ \frac{1}{S_m(V_{\vec{\bullet}};Y)},\ \delta_{Z\cap Y,m}(Y;W_{\vec{\bullet}},\mathcal{F}) \right\}$,
    then by definition we know that $(Y, \lambda_m \Gamma\vert_{Y})$ is lc at some point $p\in Z\cap Y$. By inversion of adjunction, the pair $(X, Y+\lambda_m \Gamma)$ is lc at $p$. Again by the definition of $\lambda_m$ we have $\lambda_m S_m(V_{\vec{\bullet}};Y)\leq 1$, hence the pair $$\left(X, \lambda_m\left(S_m(V_{\vec{\bullet}};Y)\cdot Y+\Gamma\right)\right)=(X, \lambda_m D)$$ is lc at $p$. Since $D$ is arbitrary, we can conclude that $$\delta_{Z\cap Y,m}(X;V_{\vec{\bullet}},\mathcal{F}_1)\geq \lambda_m=\min \left\{ \frac{1}{S_m(V_{\vec{\bullet}};Y)}, \delta_{Z\cap Y,m}(Y;W_{\vec{\bullet}},\mathcal{F}) \right\}.$$

    When taking $\limsup$ on both sides, we notice $\limsup$ commutes with $\min$ when one of the terms in the right-hand side has limit. Since $$\lim_{m\rightarrow\infty}\frac{1}{S_m(V_{\vec{\bullet}};Y)}=\frac{1}{S(V_{\vec{\bullet}};Y)}$$ exists, we obtain the desired inequality. 
\end{proof}

Let $X$ and $V_{\vec{\bullet}}$ as above, now we consider $$Y_{\bullet}:X=Y_0\supset Y_1\supset \dots \supset Y_{\ell}$$ an admissible flag of length $\ell$ on $X$. Assume that each $Y_i$ in the flag has at worst klt singularity and is a Cartier divisor on $Y_{i-1}$. Let $Y_{\bullet}^{(j)}$ be the flag given by $Y_0\supset Y_1\supset \dots \supset Y_j$, and let $W_{\vec{\bullet}}^{(j)}$ be the refinement of $V_{\vec{\bullet}}$ by $Y_{\bullet}^{(j)}$. Let $\mathcal{F}=\mathcal{F}_0$ be an arbitrary filtration on $W_{\vec{\bullet}}^{(\ell)}$, indexed by some totally ordered abelian monoid, we can inductively lift it to filtrations $\mathcal{F}_{\ell-i}$ on $W_{\vec{\bullet}}^{(i)}$ for all $0\leq i \leq \ell-1$. 
\begin{theorem}\label{AZ Admissible Flag}
    Let $Z\subset X$ be an irreducible closed subset of $X$ that intersects $Y_{\ell}$, we have $$\delta_Z(X;V_{\vec{\bullet}})\geq \min \left\{ \min\limits_{0\leq i\leq \ell-1} \left\{\frac{1}{S(W_{\vec{\bullet}}^{(i)};Y_{i+1})}\right\},\ \ \delta_{Z\cap Y_{\ell}}(Y_{\ell};W_{\vec{\bullet}}^{(\ell)},\mathcal{F}) \right\}.$$
\end{theorem}
\begin{proof}
    From Example \ref{refinement}, we know that $W_{\vec{\bullet}}^{(i)}$ contains an ample series and has bounded support for all~$0 \leq i \leq \ell$.
    Denote $W_{\vec{\bullet}}^{(0)}=V_{\vec{\bullet}}$. Applying Lemma \ref{lemma: induction step of AZ} to each $0\leq i\leq \ell-1$, we have $$\delta_{Z\cap Y_i}(Y_i;W^{(i)}_{\vec{\bullet}},\mathcal{F}_{\ell-i})\geq \min \left\{ \frac{1}{S(W_{\vec{\bullet}}^{(i)};Y_{i+1})}, \delta_{Z\cap Y_{i+1}}(Y_{i+1};W_{\vec{\bullet}}^{(i+1)},\mathcal{F}_{\ell-i-1}) \right\}.$$
    Note $\delta_{Z\cap Y_0}(Y_0;W^{(0)}_{\vec{\bullet}},\mathcal{F}_{\ell})=\delta_Z(X; V_{\vec{\bullet}})$, we then finish the proof by induction on $i$.
\end{proof}

\section{Estimate of Local Stability Threshold}\label{section: estimate of local stability threshold}
As mentioned in the introduction, the following criterion guides the proof of Theorem A.
\begin{lemma}(\autocite[Lemma 4.11]{AZ1})\label{criterionInAZ}
    Let $X$ be a Fano manifold of dimension $n$. Assume that

    1. $\delta_Z(X)\geq \frac{n+1}{n}$ for any subvariety $Z\subset X$ of dimension $\geq 1$;

    2. $\beta_X(E_x)>0$ for any $x\in X$, where $E_x$ denotes the exceptional divisor of the ordinary blowup of $x$.

    Then X is K-stable.
\end{lemma}

To justify the first condition in Lemma \ref{criterionInAZ}, we firstly notice that it is sufficient to assume $Z$ is an irreducible curve. To determine $\delta_Z(X)$ in such case, we generalize a lemma in \cite{AZ1}, which estimates the local stability threshold at multiple points along a curve, by releasing some assumptions on the multigraded linear series. Then we give a convex-geometric interpretation of the term appearing in the estimate. 
    
First we need the following preparations.

\subsection{Stability Threshold Along Multiple Points}\label{Stability Threshold Along Several Points}

\begin{definition}\label{definition: f(gamma)}
    Let $C$ be a smooth curve, and $W_{\vec{\bullet}}$ an $\bb{N}^{r+1}$-graded linear series on $C$. For any $\vec{a}\in \bb{N}^{r+1}$ such that $W_{\vec{a}}\neq 0$, the fixed part $F_{\vec{a}}$ is a divisor on $C$, so $\deg F_{\vec{a}}$ is a well-defined number. 
    
    Note $F_{\vec{a}}+F_{\vec{a}^\prime}\geq F_{\vec{a}+\vec{a}^\prime}$ for any $\vec{a},\vec{a}^\prime\in\bb{N}^{r+1}$. Therefore, by Fekete's lemma, for any $\gamma\in \operatorname{Int}(\Delta_{\text{Supp}})\cap \bb{Q}^r$, we can define $$f(\gamma):=\lim\limits_{m\rightarrow\infty}\frac{\deg F_{m,m\gamma}}{m}=\inf\limits_{m}\frac{\deg F_{m,m\gamma}}{m},$$ where the limit and infimum are taken over sufficiently divisible integers $m$. 

    Furthermore, it is easy to show that $f(t\gamma_1+(1-t)\gamma_2)\leq tf(\gamma_1)+(1-t)f(\gamma_2)$ for all $t\in \bb{Q}\cap[0,1]$ and all $\gamma_1, \gamma_2\in \operatorname{Int}(\Delta_{\text{Supp}})\cap \bb{Q}^r$. Therefore, we can extend $f(\gamma)$ to a convex (hence continuous) function on~$\operatorname{Int}(\Delta_{\text{Supp}})$.
\end{definition}

\begin{lemma}\label{DegreeOfFixed}
    Let $C$ be a smooth curve, let $W_{\vec{\bullet}}$ be an $\bb{N}^{r+1}$-graded linear series on $C$ that has bounded support, and contains an ample series. Then $$\lim_{m\rightarrow\infty}\deg F_m(W_{\vec{\bullet}})=\frac{(r+1)!}{\operatorname{vol}(W_{\vec{\bullet}})}\int_{\Delta_{\operatorname{Supp}}}\operatorname{vol}_{W_{\vec{\bullet}}}(\gamma)f(\gamma)d\gamma.$$ In particular, it is a finite real number.
\end{lemma}
\begin{proof}
    Recall 
    $$F_m=F_m(W_{\vec{\bullet}}):=\frac{1}{mh^0(W_{m,\vec{\bullet}})}\sum_{\vec{a}\in \mathbb{N}^r}\dim W_{m,\vec{a}}F_{m,\vec{a}},$$
    then
    $$\deg F_m=\frac{m^{r+1}}{h^0(W_{m,\vec{\bullet}})}\sum_{\gamma} \frac{\dim W_{m,m\gamma}}{m^{r+1}}\frac{\deg F_{m,m\gamma}}{m},$$
    where the sum is taken over all $\gamma\in(\frac{1}{m}\bb{N}^{r})\cap \Delta_{\operatorname{Supp}(W_{\vec{\bullet}})}$.

    We define the function $f_m : \Delta_{\text{Supp}} \rightarrow \bb{R}_{\geq 0}$ by
    $$f_m(\gamma)=\left\{
    \begin{array}{lr}
        \frac{\deg F_{m,m\gamma}}{m}\quad &\text{if\ } W_{m,m\gamma}\neq 0,  \\
        0 \quad & \text{otherwise. } 
        \end{array}
    \right.$$

    Consider the admissible flag $C\ni p$ given by an arbitrary closed point $p$ on $C$, let $\Delta=\Delta(W_{\vec{\bullet}})$ be the Okounkov body of $W_{\vec{\bullet}}$ with respect to this flag.
    It is clear that $$\deg F_m=\frac{\int (f_m\circ p) d\rho_m}{\int d\rho_m},$$ where $p: \Delta\rightarrow \Delta_{\text{Supp}}$ is the natural projection. 

    We know that the measure $\rho_m$ weakly converges to the Lebesgue measure $\rho$ on $\Delta$, and~$\limsup_{m\rightarrow\infty} f_m(\gamma)=f(\gamma)$ for all $\gamma\in \operatorname{Int}(\Delta_{\text{Supp}})\cap \bb{Q}^r$. Here $f(\gamma)$ is defined in Definition \ref{definition: f(gamma)}. Next, we will show that $$\lim_{m\rightarrow\infty} \deg F_m=\frac{\int (f\circ p) d\rho}{\operatorname{vol}(\Delta)}=\frac{(r+1)!}{\operatorname{vol}(W_{\vec{\bullet}})}\int_{\Delta_{\operatorname{Supp}}}\operatorname{vol}_{W_{\vec{\bullet}}}(\gamma)f(\gamma)d\gamma,$$ and this completes our proof.

    Assume $W_{\vec{\bullet}}$ is a multigraded linear series associated to $L_0,\dots, L_r$, and denote $\vec{L}=(L_1, \dots, L_r)$. Since~$W_{\vec{\bullet}}$ has bounded support, there exists some constant $C_1>0$ such that $W_{m,\vec{a}}\neq 0$ only if $\lvert \vec{a} \rvert\leq C_1m$. Thus as $mL_0+\vec{a}\cdot \vec{L}-F_{m,\vec{a}}$ is effective, we know $\deg F_{m,\vec{a}}\leq \deg (mL_0+\vec{a}\cdot \vec{L})\leq Tm$ for some constant~$T>0$. Therefore $0\leq f(\gamma), f_m(\gamma)\leq T$ for all $m, \gamma$.

    Fix $\varepsilon>0$ and $\gamma\in \operatorname{Int}(\Delta_{\text{Supp}})$, then we can choose $\gamma_i\in \operatorname{Int}(\Delta_{\text{Supp}})\cap \bb{Q}^r, i=0,\dots ,r$, such that the convex hull spanned by them contains $\gamma$ in the interior, and $f(\gamma_i)< f(\gamma)+\varepsilon$. Fix a sufficiently divisible~$m_0$ such that $W_{m_0,m_0\gamma_i}\neq 0$ and $f_{m_0}(\gamma_i)<f(\gamma)+\varepsilon$. Then we can find a small open neighborhood $U=U(\gamma)$ in the interior of the convex hull of $\gamma_i$s, such that $f_{m_0}(\gamma_i)<f(\zeta)+\varepsilon$, for all $\zeta \in U$. Let $\Pi\subset \bb{R}^{r+1}$ be the cone spanned by the vectors $(1,\gamma_i)$. According to \autocite[Lemma 1.1]{Xubook}, there exists an $\vec{a}_0\in\bb{N}^{r+1}$ such that $W_{\vec{a}}\neq 0$ for all $\vec{a}\in (\vec{a}_0+\Pi) \cap \bb{N}^{r+1}$. Then there exists constants $C_0>0, C_2>0$ such that for all~$\zeta\in U\cap\bb{Q}^r$, and for any $m\in \bb{N}$ with $m\geq C_0$ and $m\zeta\in \bb{N}^r$, we have the following: $$(m,m\zeta)=\vec{a}+\sum_{i=0}^{r}k_i(m_0,m_0\gamma_i),$$ for some $\vec{a}\in \bb{N}^{r+1}$ and $k_i\in \bb{N}$ such that $W_{\vec{a}}\neq 0$ and $\lvert\vec{a}\rvert< C_2$. Therefore, we have $W_{m,m\zeta}\neq 0$ and
    $$\deg F_{m,m\zeta}\leq \deg F_{\vec{a}}+\sum_{i=0}^{r}k_i\deg F_{m_0,m_0\gamma_i}=\deg F_{\vec{a}}+\sum_{i=0}^{r}k_im_0f_{m_0}(\gamma_i).$$
    Since $\lvert\vec{a}\rvert< C_2$, we have $\deg F_{\vec{a}}< C_2T$. Also we know $\sum_i k_im_0\leq m$, and $f_{m_0}(\gamma_i)<f(\zeta)+\varepsilon$. Then 
    $$\deg F_{m,m\zeta}\leq C_2T+m\left(f\left(\zeta\right)+\varepsilon\right).$$

    Thus, there exists $M\geq C_0$, depending on $\varepsilon$ and the open subset $U(\gamma)$, such that $f_m(\zeta)< f(\zeta)+2\varepsilon$ for all $\zeta\in U(\gamma)$ and $m>M$. Also note that both $f_m(\gamma)\leq T$ and $f(\gamma)\leq T$ for all $m,\gamma$.
 
    Now, choose a compact subset $K\Subset\operatorname{Int}(\Delta_{\text{Supp}})$ such that $\operatorname{vol}(\Delta\backslash p^{-1}(K))<\varepsilon$ and $\rho(\partial p^{-1}(K))=0$, then the weak convergence gives $\rho_m(\Delta\backslash p^{-1}(K))<2\varepsilon$ when $m\gg 1$. Since $K$ is compact, we can choose a finite subcover from the open cover $\{U(\gamma) \mid \gamma\in K\}$. Therefore there exists a uniform $M>0$ such that~$f_m(\gamma)<f(\gamma)+\varepsilon$ for all $\gamma\in K$ and $m>M$. Then
    \begin{align*}
        &\left\lvert \int (f_m\circ p)d\rho_m-\int (f\circ p)d\rho\right\rvert \leq \left\lvert \int (f_m\circ p)d\rho_m-\int (f\circ p)d\rho_m\right\rvert + \left\lvert \int (f\circ p)d\rho_m-\int (f\circ p)d\rho\right\rvert\\
        \leq &\left\lvert \int_{\Delta\backslash p^{-1}(K)} (f_m\circ p-f\circ p)d\rho_m\right\rvert +\left\lvert \int_{p^{-1}(K)} (f_m\circ p-f\circ p)d\rho_m\right\rvert + \left\lvert \int (f\circ p)d\rho_m-\int (f\circ p)d\rho\right\rvert.
    \end{align*}

    The first term is bounded by $2T\times 2\varepsilon$. For the last term, note $f\circ p$ is continuous on $p^{-1}(K)$ and the measures $\rho_m$ weakly converge to $\rho$, so the difference on $p^{-1}(K)$ tends to 0 as $m\rightarrow\infty$. Outside~$p^{-1}(K)$ the difference is bounded by $\varepsilon T$. For the middle term, note that $f(\gamma)\leq f_m(\gamma)\leq f(\gamma)+\varepsilon$ whenever~$W_{m,m\gamma}\neq~0$, and these are exactly all the points with positive measure in the atomic measure $\rho_m$. Therefore $\lvert \int_{p^{-1}(K)} (f_m\circ p-f\circ p)d\rho_m\rvert\leq \varepsilon \operatorname{vol}(\Delta).$

    We can conclude now $\lim_{m\rightarrow\infty} \int (f_m\circ p)d\rho_m=\int (f\circ p)d\rho$. 
    
    It is clear that $\lim_{m\rightarrow\infty}\int d\rho_m=\int d\rho=\operatorname{vol}(\Delta)$. This proves the first equality.

    For the second equality, note that $$\int (f\circ p)d\rho=\int_{\Delta_{\text{Supp}}} \operatorname{vol}(p^{-1}(\gamma))f(\gamma)d\gamma,$$ also we have $\operatorname{vol}(W_{\vec{\bullet}})=(r+1)!\operatorname{vol}(\Delta)$, and $\operatorname{vol}_{W_{\vec{\bullet}}}(\gamma)=\operatorname{vol}(p^{-1}(\gamma))$ for all $\gamma\in \operatorname{Int}(\Delta_{\text{Supp}})$ when $\dim C=~1$.
\end{proof}

\begin{lemma}\label{lemma: condition 2 of almost complete}
    Let $C$ be a smooth curve, and $W_{\vec{\bullet}}$ be a $\bb{N}^{r+1}$-graded linear series on $C$ that has bounded support and contains an ample series. Then $W_{\vec{\bullet}}$ satisfies the condition (2) of Definition \ref{almost complete}. 
\end{lemma}
\begin{proof}
    Assume $W_{\vec{\bullet}}$ is a multigraded linear series associated to $\vec{L}=(L_0, \dots, L_r)$. Since $W_{\vec{\bullet}}$ contains an ample series, from \autocite[Lemma 4.16]{lazarsfeld2009convex}, we see that for any $\vec{a}\in \operatorname{Int}(\operatorname{Supp}(W_{\vec{\bullet}}))\cap \bb{N}^{r+1}$, the condition 3 of containing an ample series is satisfied. That is, there exists a decomposition $\vec{a}\cdot \vec{L}=A+E$ with $A$ an ample $\Q$-line bundle and $E$ an effective $\Q$-divisor such that $H^0(C, mA)\subseteq W_{m\vec{a}}$ for all sufficiently divisible $m$.

    For all $\gamma\in\operatorname{Int}(\Delta_{\text{Supp}})\cap\Q^{r}$, we consider the singly-graded linear series $V_{\bullet}=\bigoplus_m W_{m,m\gamma}$. Let $\bb{N}(V_{\bullet})=~\{m\mid V_m\neq 0\}$. For every $m\in \bb{N}(V_{\bullet})$, decompose $V_m$ into movable and fixed part, we have $V_m=M_{m,m\gamma}+F_{m,m\gamma}$, with $M_{m, m\gamma}\subseteq \lvert (m,m\gamma)\cdot \vec{L}-F_{m,m\gamma}\rvert$. We denote~$\deg(M_{m, m\gamma}):=\deg\left(\left(m,m\gamma\right)\cdot \vec{L}-F_{m, m\gamma}\right)$. Since $\operatorname{Supp}(F_{m, m\gamma})$ is always finite, the linear system $V_m$ defines a rational map $\varphi_m\colon C \dashrightarrow \bb{P}(V_m^*)$. For $m$ sufficiently divisible, we have decomposition~$(1,\gamma)\cdot \vec{L}=A+E$ such that $mA$ is very ample and $H^0(C,mA)\subseteq V_m$. Therefore $\varphi_m$ defines a birational map, gives $\dim\operatorname{Im}(\varphi_m)=1$, and $\deg (\varphi_m(C))=\deg M_{m, m\gamma}$. Following the notation of \autocite[\S 8.1]{szekelyhidi2015filtrations}, we denote $$\kappa=\kappa(V_{\bullet}):=\operatorname{tr.deg}(\oplus_mV_m/\bb{C})-1.$$Then by \autocite[Theorem 16 (ii)]{szekelyhidi2015filtrations}, we have $\kappa=\dim\varphi_m$ for all $m\in\bb{N}$ sufficiently large, which gives $\kappa=1$. 
    Apply \autocite[Theorem 16 (i)]{szekelyhidi2015filtrations}, we know that $$\quad e=e(V_{\bullet}):=\lim_{m\in\bb{N}(V), m\rightarrow\infty}\frac{\kappa!}{m^{\kappa}}\dim V_m=\lim_{m\in\bb{N}(V), m\rightarrow\infty}\frac{\dim V_m}{m}$$ exists in $(0,\infty)$, and from the other claim in \autocite[Theorem 16 (ii)]{szekelyhidi2015filtrations} we have $$e=\lim_{m\in\bb{N}(V), m\rightarrow\infty}\frac{\deg \varphi_m(C)}{m}=\lim_{m\in\bb{N}(V), m\rightarrow\infty}\frac{\deg M_{m, m\gamma}}{m}.$$ Therefore for any closed point $p\in C$, we have 
    \begin{equation}\label{eq1}
        1=\lim_{m\rightarrow\infty}\frac{\dim V_m}{\deg M_{m, m\gamma}}=\lim_{m\rightarrow\infty}\frac{\dim M_{m, m\gamma}}{\deg M_{m, m\gamma}}=\lim_{m\rightarrow\infty}\frac{\dim M_{m, m\gamma}}{h^0(C,\deg M_{m,m\gamma}\cdot p)},
    \end{equation}where the limit is taken over sufficiently divisible $m$. The last equality is due to the Riemann-Roch theorem for smooth curves.
\end{proof}

With Lemmas \ref{DegreeOfFixed} and \ref{lemma: condition 2 of almost complete}, we can generalize \autocite[Lemma 4.16]{AZ1} to the following:

\begin{lemma}\label{GeneralizedLocalDeltaOnCurve}
    Let $C$ be a smooth curve, and $W_{\vec{\bullet}}$ be a $\bb{N}^{r+1}$-graded linear series on $C$ that has bounded support and contains an ample series. Let $Z=\{p_1,\dots, p_k\}$ be a collection of $k$ distinct points on $C$, let~$Q\in C$ be a very general point on $C$, and $\mathcal{F}$ be the filtration on $W_{\vec{\bullet}}$ induced by $Q$. Then
    $$\delta_{Z}(C;W_{\vec{\bullet}},\mathcal{F})\geq \frac{2k}{\deg c_1(W_{\vec{\bullet}})+\lim_{m\rightarrow \infty}\deg F_{m}}.$$
    The exact meaning of ``very general" will be made explicit in the proof.
\end{lemma}
\begin{proof}
    Following the proof of \autocite[Lemma 4.16]{AZ1}, we have 
    $$\delta_Z(C; W_{\vec{\bullet}},\mathcal{F})\geq \frac{k}{\deg c_1(W_{\vec{\bullet}})-S(W_{\vec{\bullet}};Q)}.$$
    Therefore the only thing left is to compute $S(W_{\vec{\bullet}};Q)$.
    
    Since for every $m\geq 0$, $F_m(W_{\vec{\bullet}})$ has finite support, the union $\bigcup_{m\geq 0}\operatorname{Supp}\left(F_m\left(W_{\vec{\bullet}}\right)\right)$ is a countable set. We take $Q\in C$ an arbitrary point avoiding this set, and let $\mathcal{F}$ be the filtration induced by $Q$. We can define the multigraded linear series $W_{\vec{\bullet}}^t$ associated to $\mathcal{F}$ (see Remark \ref{volume function and compute S}). Following the notation in the proof of Lemma \ref{lemma: condition 2 of almost complete}, set $$g(\gamma):=\lim_{m\rightarrow\infty}\frac{\deg M_{m,m\gamma}}{m}=\deg\left(\left(1,\gamma\right)\cdot\vec{L}\right)-f(\gamma),$$ here $f(\gamma)$ is defined in Definition \ref{definition: f(gamma)}.
    Using the same argument as in the proof of \autocite[Lemma 2.13]{AZ1}, we have $$\lim_{m\rightarrow\infty}\frac{\dim \mathcal{F}^{mt}W_{m,m\gamma}}{m}=\lim_{m\rightarrow\infty}\frac{\dim \mathcal{F}^{mt}H^0(C,\deg M_{m,m\gamma}\cdot P)}{m}$$ for any point $P\in C$. Although almost completeness is assumed in \autocite[Lemma 2.13]{AZ1}, we emphasize that, the only property of $W_{\vec{\bullet}}$ used in the relevant part of that proof is equation (\ref{eq1}), which we have already established.
    
    Recall Remark \ref{volume function and compute S}, the left-hand side is exactly $\operatorname{vol}_{W_{\vec{\bullet}}^t}(\gamma)$. We compute the right-hand side by Riemann-Roch theorem: 
    \begin{align*}
        \lim_{m\rightarrow\infty}&\frac{\dim \mathcal{F}^{mt}H^0(C,\deg M_{m,m\gamma}\cdot P)}{m}=\lim_{m\rightarrow\infty}\frac{h^0(C,\deg M_{m,m\gamma}\cdot P-mt\cdot Q)}{m}\\
        &=\lim_{m\rightarrow\infty}\frac{\deg M_{m,m\gamma}-mt}{m}=g(\gamma)-t.
    \end{align*}
    Therefore, $\operatorname{vol}_{W_{\vec{\bullet}}^t}(\gamma)=g(\gamma)-t$ for all $t\in [0,g(\gamma))$. Then 
    \begin{align*}
        S(W_{\vec{\bullet}};\mathcal{F})&=\frac{1}{\operatorname{vol}\left(\Delta\left(W_{\vec{\bullet}}\right)\right)}\int_0^{\infty}\int_{\Delta_{\text{Supp}}^t}\operatorname{vol}_{W_{\vec{\bullet}}^t}(\gamma)d\gamma dt =\frac{1}{\operatorname{vol}\left(\Delta\left(W_{\vec{\bullet}}\right)\right)}\int_{\Delta_{\text{Supp}}}\frac{\left(g(\gamma)\right)^2}{2}d\gamma\\
        &=\frac{1}{2\operatorname{vol}\left(\Delta\left(W_{\vec{\bullet}}\right)\right)}\int_{\Delta_{\text{Supp}}}g(\gamma)\cdot \deg\left((1,\gamma)\cdot \vec{L}\right)d\gamma-\frac{1}{2\operatorname{vol}\left(\Delta\left(W_{\vec{\bullet}}\right)\right)}\int_{\Delta_{\text{Supp}}}g(\gamma)f(\gamma)d\gamma.
    \end{align*}
    Recall that $$\lim_{m\rightarrow\infty}\deg F_m(W_{\vec{\bullet}})=\frac{(r+1)!}{\operatorname{vol}(W_{\vec{\bullet}})}\int_{\Delta_{\operatorname{Supp}}}\operatorname{vol}_{W_{\vec{\bullet}}}(\gamma)f(\gamma)d\gamma,$$ and note $\operatorname{vol}_{W_{\vec{\bullet}}}(\gamma)=g(\gamma), \operatorname{vol}(W_{\vec{\bullet}})=(r+1)!\operatorname{vol}\left(\Delta\left(W_{\vec{\bullet}}\right)\right)$. Also notice that (or refer to proof of \autocite[Lemma-Definition 2.14]{AZ1}) $$\deg c_1(W_{\vec{\bullet}})=\deg L_0+\sum_{i=1}^r\frac{1}{\operatorname{vol}\left(\Delta\left(W_{\vec{\bullet}}\right)\right)}\int_{\Delta\left(W_{\vec{\bullet}}\right)}x_id\rho\cdot \deg L_i=\frac{1}{\operatorname{vol}\left(\Delta\left(W_{\vec{\bullet}}\right)\right)}\int_{\Delta_{\operatorname{Supp}}}g(\gamma)\cdot\deg\left(\left(1,\gamma\right)\cdot \vec{L} \right)d\gamma.$$
    Therefore $S(W_{\vec{\bullet}};\mathcal{F})=\frac{1}{2}\left(\deg c_1\left(W_{\vec{\bullet}}\right)-\lim_{m\rightarrow\infty}\deg F_m\left(W_{\vec{\bullet}}\right)\right)$. This completes the proof.
\end{proof}

\begin{definition}\label{definition: construction of the convex bodt}
    Let $C$ be a smooth curve, and $W_{\vec{\bullet}}$ be an $\bb{N}^{r+1}$-graded linear series on $C$ associated to~$\bb{Q}$-Cartier $\Q$-divisors $L_0, L_1, \dots, L_r$. Denote $\alpha=\deg L_0, \beta_i=\deg L_i (1\leq i\leq r)$ and $\vec{\beta}=(\beta_1, \dots, \beta_r)$. Assume $W_{\vec{\bullet}}$ has bounded support and contains an ample series, we associate to $W_{\vec{\bullet}}$ a bounded convex body $\bb{O}(W_{\vec{\bullet}})\subset \bb{R}^{r+1}$ as follows.

    Consider the discrete semigroup $$\Gamma=\left\{(m, \vec{a}, x)\in \bb{N}^{r+2}\mid W_{m,\vec{a}}\neq 0, x\in\left[\deg F(W_{m,\vec{a}}), \alpha m+\vec{a}\cdot \vec{\beta}\right]\cap \bb{N} \right\},$$
    then let $\Sigma$ be the closed convex cone spanned by $\Gamma$, and $\bb{O}(W_{\vec{\bullet}})=\Sigma\cap \left(\{1\}\times \bb{R}^{r+1}\right)\subset \bb{R}^{r+1}$. 

    Explicitly, we can write 
    \begin{equation}\label{construction of O(W)}
        \bb{O}(W_{\vec{\bullet}})=\overline{\left\{(\gamma,z)\in\bb{R}^r\times \bb{R}\mid \gamma\in\operatorname{Int}\Delta_{\operatorname{Supp}(W_{\vec{\bullet}})}, f(\gamma)\leq z\leq \alpha+\vec{\beta}\cdot\gamma \right\}}\subset \bb{R}^{r+1}.    
    \end{equation}
    It is bounded as $W_{\vec{\bullet}}$ has bounded support, and is convex as $f(\gamma)$ is convex.
\end{definition}

\begin{proposition}\label{proposition: Okounkov body v.s O(W)}
    Pick $Q\in C$ a very general point, in the sense of Lemma \ref{GeneralizedLocalDeltaOnCurve}, then the Okounkov body of $W_{\vec{\bullet}}$ with respect to the admissible flag $C\ni Q$ is: 
    $$\Delta(W_{\vec{\bullet}})=\left\{(\gamma,z)\in \bb{R}^r\times \bb{R}\mid \gamma\in\Delta_{\operatorname{Supp}(W_{\vec{\bullet}})}, 0\leq z\leq \alpha+\vec{\beta}\cdot\gamma-f(\gamma) \right\}.$$ 
    In particular, we have $\operatorname{vol}(\Delta(W_{\vec{\bullet}}))=\operatorname{vol}(\bb{O}(W_{\vec{\bullet}}))$, and the first $r$ coordinates of the barycenters of these two convex bodies are identical.
\end{proposition}

Informally speaking, the convex body $\bb{O}(W_{\vec{\bullet}})$ simply ``moves the last coordinate of the Okounkov body upward by $f(\gamma)$".

\begin{theorem}\label{read info from barycenter of convex body}
    Notation as above. Let $(b_1,\dots, b_{r+1})$ be the barycenter of $\bb{O}(W_{\vec{\bullet}})$, then $$\deg c_1(W_{\vec{\bullet}})=\alpha+\sum_{i=1}^r \beta_i b_i,\quad \deg c_1(W_{\vec{\bullet}})+\lim_{m\rightarrow \infty}\deg F_{m}=2b_{r+1}.$$
\end{theorem}
\begin{proof}
   Let $Q$ be a very general point on $C$, let $W_{\vec{\bullet}}^Q$ be the refinement of $W_{\vec{\bullet}}$ by $Q$, and $\Delta(W_{\vec{\bullet}})$ be the Okounkov body of $W_{\vec{\bullet}}$ with respect to the admissible flag $C\ni Q$ as in Proposition \ref{proposition: Okounkov body v.s O(W)}. We write~$\Delta=~\Delta(W_{\vec{\bullet}}), \Delta_{\operatorname{Supp}}=\Delta_{\operatorname{Supp}(W_{\vec{\bullet}})}$ and $\bb{O}=\bb{O}(W_{\vec{\bullet}})$ for simplicity. Then we know that $\operatorname{vol}(\Delta)=\operatorname{vol}(\bb{O})$.
   
    By definition we know that $$c_1(W_{\vec{\bullet}})=\lim_{m\rightarrow \infty}\frac{\sum_{\vec{a}\in\bb{N}^{r}}[mL_0+\vec{a}\cdot \vec{L}]\cdot h^0(W_{m,\vec{a}})}{mh^0(W_{m,\vec{\bullet}})}\in \operatorname{Pic}(C)_{\bb{R}},$$ where $\vec{L}=(L_1, \dots, L_r)$. 

    Denote $\Gamma=\{\vec{a}\in\bb{N}^{r+1}\mid W_{\vec{a}}\neq 0\}$ and $\Gamma_m=\Gamma\cap (\{m\}\times \bb{N}^{r})$. We can rewrite the equation as follows.
    \begin{equation}\label{degree of c_1}
    \begin{aligned}
        \deg c_1(W_{\vec{\bullet}})&=\lim_{m\rightarrow \infty}\frac{\sum_{\vec{a}\in\bb{N}^{r}}(m\alpha+\vec{a}\cdot \vec{\beta})\cdot h^0(W_{m,\vec{a}})}{mh^0(W_{m,\vec{\bullet}})}\\
        &=\lim_{m\rightarrow\infty}\sum_{\gamma\in \frac{1}{m}\Gamma_m}(\alpha+\gamma\cdot \vec{\beta})\cdot h^0(W_{m,m\gamma})\cdot \frac{m^{r+1}}{h^0(W_{m,\vec{\bullet}})}\cdot \frac{1}{m^{r+1}}\\
        &=\frac{(r+1)!}{\operatorname{vol}(W_{\vec{\bullet}})}\lim_{m\rightarrow\infty}\frac{1}{m^{r+1}}\sum_{\gamma\in \frac{1}{m}\Gamma_m}(\alpha+\gamma\cdot \vec{\beta})\cdot h^0(W_{m,m\gamma})=\frac{1}{\operatorname{vol}(\Delta)}\lim_{m\rightarrow\infty}\int_{\Delta}(\alpha+\gamma\cdot \vec{\beta}) d\rho_m\\
        &=\frac{1}{\operatorname{vol}(\Delta)}\left(\int_{\Delta}\alpha d\rho + \sum_{i=1}^r\int_{\Delta}\beta_ix_i d\rho\right)=\alpha+\sum_{i=1}^r\beta_i b_i.
    \end{aligned}
    \end{equation}
    The last equality holds because of the observation that the first $r$ coordinates of the barycenters of $\Delta$ and~$\bb{O}$ are the same.
    
    For any $\gamma\in\operatorname{Int}(\Delta_{\operatorname{Supp}})\cap \bb{Q^r}$, we have the decomposition $W_{m,m\gamma}=M_{m,m\gamma}+F_{m,m\gamma}$ of $W_{m,m\gamma}$ into movable and fixed parts, for all sufficiently divisible $m$. We know $M_{m,m\gamma}\subset \lvert mL_0+\sum_{i=1}^r m\gamma_iL_i-F_{m,m\gamma}\rvert$. Denote the degree of this divisor by $\deg M_{m,m\gamma}$. Then, from the proof of Lemma \ref{lemma: condition 2 of almost complete}, we see 
    \begin{align*}
        \deg F_{m,m\gamma}&=m\alpha+m\gamma\cdot\vec{\beta}-\deg M_{m,m\gamma}=m\alpha+m\gamma\cdot\vec{\beta}-\dim M_{m,m\gamma}+o(m)\\
        &=m\alpha+m\gamma\cdot\vec{\beta}-\dim W_{m,m\gamma}+o(m). 
    \end{align*}
    Therefore $$f(\gamma)=\lim_{m\rightarrow\infty}\frac{\deg F_{m,m\gamma}}{m}=\alpha+\gamma\cdot \vec{\beta}-\operatorname{vol}_{W_{\vec{\bullet}}}(\gamma).$$ From Definition \ref{definition: f(gamma)}, we see $f(\gamma)$ can be extended to a convex function on $\operatorname{Int}(\Delta_{\operatorname{Supp}})$. For simplicity we denote $u(\gamma)=\operatorname{vol}_{W_{\vec{\bullet}}}(\gamma)$. Then from Lemma \ref{DegreeOfFixed} we have
    \begin{equation}\label{degree of F_m}
    \begin{aligned}
        \lim_{m\rightarrow\infty}&\deg F_m\left(W_{\vec{\bullet}}\right)=\frac{(r+1)!}{\operatorname{vol}(W_{\vec{\bullet}})}\int_{\Delta_{\operatorname{Supp}}} u(\gamma)f(\gamma)d\gamma=\frac{(r+1)!}{\operatorname{vol}(W_{\vec{\bullet}})}\int_{\Delta_{\operatorname{Supp}}}\left(\alpha+\gamma\cdot \vec{\beta}-f\left(\gamma\right)\right)\cdot f(\gamma)d\gamma\\
        &=\frac{(r+1)!}{\operatorname{vol}(W_{\vec{\bullet}})}\int_{\Delta_{\operatorname{Supp}}} \left(\alpha+\gamma\cdot \vec{\beta}\right)f(\gamma)-\left(f(\gamma)\right)^2 d\gamma\\
        &=\frac{(r+1)!}{\operatorname{vol}(W_{\vec{\bullet}})}\int_{\Delta_{\operatorname{Supp}}} \left(\alpha+\gamma\cdot\vec{\beta}\right)\cdot \left(\alpha+\gamma\cdot\vec{\beta}-u\left(\gamma\right)\right) -\left(f(\gamma)\right)^2 d\gamma\\
        &=\frac{1}{\operatorname{vol}(\bb{O})}\int_{\Delta_{\operatorname{Supp}}} \left(\alpha+\gamma\cdot\vec{\beta}\right)^2-\left(\alpha+\gamma\cdot\vec{\beta}\right)\cdot u(\gamma)-\left(f(\gamma)\right)^2d\gamma\\
        &=\frac{2}{\operatorname{vol}(\bb{O})}\int_{\Delta_{\operatorname{Supp}}} \frac{1}{2} \left(\left(\alpha+\gamma\cdot\vec{\beta}\right)^2-\left(f\left(\gamma\right)\right)^2\right)d\gamma-\frac{\alpha}{\operatorname{vol}(\bb{O})}\int_{\Delta_{\operatorname{Supp}}}u(\gamma)d\gamma-\sum_{i=1}^r \frac{\beta_i}{\operatorname{vol}(\bb{O})}\int_{\Delta_{\operatorname{Supp}}}\gamma_i u(\gamma)d\gamma\\
        &=2b_{r+1}-\alpha-\sum_{i=1}^r\beta_i b_i.
    \end{aligned}
    \end{equation}

    The last equality holds since by definition 
    $$b_{r+1}=\frac{1}{\operatorname{vol}(\bb{O})}\int_{\bb{O}}z dx=\frac{1}{\operatorname{vol}(\bb{O})}\int_{\Delta_{\operatorname{Supp}}}\left(\int_{f(\gamma)}^{\alpha+\vec{\beta}\cdot \gamma}z dz\right)d\gamma=\frac{1}{\operatorname{vol}(\bb{O})}\int_{\Delta_{\operatorname{Supp}}}\frac{1}{2} \left(\left(\alpha+\gamma\cdot\vec{\beta}\right)^2-f^2\left(\gamma\right)\right)d\gamma,$$ and $$b_i=\frac{1}{\operatorname{vol}(\bb{O})}\int_{\bb{O}}x_i dx=\frac{1}{\operatorname{vol}(\bb{O})}\int_{\Delta_{\operatorname{Supp}}}\left(\int_{f(\gamma)}^{\alpha+\vec{\beta}\cdot \gamma}\gamma_i dz\right)d\gamma=\frac{1}{\operatorname{vol}(\bb{O})}\int_{\Delta_{\operatorname{Supp}}}\gamma_iu(\gamma)d\gamma$$ for all $1\leq i\leq r$. Here we denote $x=(\gamma, z)=(\gamma_1, \dots, \gamma_r, z)$ a point in $\bb{O}(W_{\vec{\bullet}})$.

    Then adding up equations (\ref{degree of c_1}) and (\ref{degree of F_m}), we see that $\deg c_1(W_{\vec{\bullet}})+\lim_{m\rightarrow \infty}\deg F_{m}=2b_{r+1}$.
\end{proof}

\subsection{Stability Threshold Along Positive Dimensional Closed Subset}\label{subsection: stability threshold along positive dimensional subset}
In this subsection we see how Subsection \ref{Stability Threshold Along Several Points} can be applied to justify the first condition of Lemma \ref{criterionInAZ}. Throughout this subsection, we consider the following settings. 

Let $X$ be a smooth variety. Let $L$ be an ample line bundle on $X$ and $V_{\bullet}=\oplus_{m}H^0(X,mL)$ be the complete linear series associated to $L$. Consider the admissible flag $$Y_{\bullet}: X=Y_0\supset Y_1\supset \dots \supset Y_{n-1}$$ on $X$ such that each $Y_i$ is normal, and is a Cartier divisor on $Y_{i-1}$ and $Y_{n-1}$ is a smooth curve. Let $W_{\vec{\bullet}}^{(i)}$ be the refinement of $V_{\bullet}$ by $Y_{\bullet}^{(i)} : X=Y_0\supset Y_1\supset\cdots \supset Y_i$ for $1\leq i\leq n-1$ and denote $W_{\vec{\bullet}}^{(0)}=V_{\bullet}$. Then each $W_{\vec{\bullet}}^{(i)}$ is an $\bb{N}\times \bb{N}^{i}$-graded linear series on $Y_i$, it has bounded support and contains an ample series. Denote $\bb{O}=\bb{O}(W_{\vec{\bullet}}^{(n-1)})$ and let $(b_1,\dots, b_n)$ be its barycenter.

We can combine the Abban--Zhuang method in Theorem \ref{AZ Admissible Flag} and the previous theorem to obtain the following.

\begin{corollary}\label{abban-zhuang in terms of barycenter}
    With the above convention, let $Z\subset X$ be an irreducible closed subset such that the set-theoretic intersection $Z\cap Y_{n-1}$ consists of $k$ distinct closed points. Then we have 
    \begin{align}\label{abban-zhuang stops at a curve}
        \delta_Z(X;L)\geq \min\left\{\frac{1}{b_1},\dots, \frac{1}{b_{n-1}},\frac{k}{b_n}\right\}.
    \end{align}
\end{corollary}
\begin{proof}
    Let $Q\in Y_{n-1}$ be a very general point on $Y_{n-1}$, in the sense of Lemma \ref{GeneralizedLocalDeltaOnCurve}. Explicitly, let $Q\not\in \operatorname{Supp}F_m(W_{\vec{\bullet}}^{(n-1)})$ for any $m\in \bb{N}$ and $Q\not\in Z\cap Y_{n-1}$. Let $\mathcal{F}$ be the filtration on $W_{\vec{\bullet}}^{(n-1)}$ induced by $Q$. We inductively lift $\mathcal{F}$ to each $W_{\vec{\bullet}}^{(i)}, 0\leq i\leq n-2$, and also denote the filtration by $\mathcal{F}$. For the lift of filtration, see Equation (\ref{eq: lift of filtration}). Then from Theorem \ref{AZ Admissible Flag} we know that $$\delta_Z(X;L)=\delta_Z(X;V_{\bullet},\mathcal{F})\geq \min\left\{ \min_{0\leq i\leq n-2}\left\{  \frac{A_{Y_i}(Y_{i+1})}{S(W_{\vec{\bullet}}^{(i)};Y_{i+1})} \right\}, \delta_{Z\cap Y_{n-1}}(Y_{n-1};W_{\vec{\bullet}}^{(n-1)},\mathcal{F})   \right\}.$$

    Firstly we look at the first $n-1$ terms. Since each $Y_{i+1}$ is a divisor on $Y_i$, the log discrepancy~$A_{Y_i}(Y_{i+1})=1$. We now show that $S(W_{\vec{\bullet}}^{(i)};Y_{i+1})=b_{i+1}$, for all $0\leq i\leq n-2$.

    From the definition, it is not hard to see that $S_m(W_{\vec{\bullet}}^{(i)};Y_{i+1})=\operatorname{ord}_{Y_{i+1}}(D)$ for an arbitrary $m$-basis type $\Q$-divisor $D$ of $W_{\vec{\bullet}}^{(i)}$ that is compatible with the filtration on $W_{\vec{\bullet}}^{(i)}$ induced by $Y_{i+1}$. Such $\Q$-divisor $D$ is of the form $D=\frac{1}{mh^0(W_{m,\bullet}^{(i)})}\sum_{\vec{a}\in\bb{N}^{i}} D_{\vec{a}}$, where $D_{\vec{a}}=\sum_{j=1}^{\dim W_{m,\vec{a}}^{(i)}} \operatorname{div}(s_j)$ such that $\{s_1, \dots, s_{\dim W_{m,\vec{a}}^{(i)}}\}$ form a basis of $W_{m,\vec{a}}^{(i)}$ that is compatible with the filtration induced by $Y_{i+1}$. There are exactly $\dim W_{m,\vec{a},j}^{(i+1)}$ of them whose vanishing order along $Y_{i+1}$ is $j$, for any $j\in\bb{N}$. Therefore $\operatorname{ord}_{Y_{i+1}}(D_{\vec{a}})=\sum_{j=0}^{\infty}j\cdot \dim W_{m,\vec{a},j}^{(i+1)}$. Note this is actually a finite sum since $W_{\vec{\bullet}}^{(i+1)}$ has bounded support. Then we have 
    \begin{align*}
        S_m(W_{\vec{\bullet}}^{(i)};Y_{i+1})&=\operatorname{ord}_{Y_{i+1}}(D)=\frac{1}{mh^0(W_{m,\bullet}^{(i)})}\sum\limits_{\vec{a}\in\bb{N}^{i}}\sum\limits_{j=0}^{\infty}j\cdot \dim W_{m,\vec{a},j}^{(i+1)}\\
        &=\frac{1}{mh^0(W_{m,\bullet}^{(i+1)})}\sum_{\vec{a}\in\bb{N}^{i+1}}a_{i+1}\cdot \dim W^{(i+1)}_{m,\vec{a}},
    \end{align*} 
    where $a_{i+1}$ is the $(i+1)$-th coordinate of the vector $\vec{a}\in\bb{N}^{i+1}$. For the last equation we also use the fact~$h^0(W_{m,\bullet}^{(j)})=\dim V_m$ for all $0\leq j\leq n-1$.
    
    Next, note that $\dim W^{(i+1)}_{m,\vec{a}}=\sum_{\vec{b}\in \bb{N}^{\ell -i-1}}\dim W^{(\ell)}_{m,\vec{a},\vec{b}}$ for any $\vec{a}\in\bb{N}^{i+1}$ and any $\ell\geq i+1$. If $i+1\leq n-1$, taking $\ell=n-1$, we derive
    \begin{align*}
        S_m(W_{\vec{\bullet}}^{(i)};Y_{i+1})&=\frac{1}{mh^0(W_{m,\bullet}^{(i+1)})}\sum_{\vec{a}\in\bb{N}^{i+1}}a_{i+1}\cdot \left(\sum_{\vec{b}\in \bb{N}^{n-i-2}}\dim W^{(n-1)}_{m,\vec{a},\vec{b}}\right)\\
        &=\frac{1}{mh^0(W_{m,\bullet}^{(n-1)})}\sum_{\vec{a}\in\bb{N}^{n-1}}a_{i+1}\cdot \dim W^{(n-1)}_{m,\vec{a}}.
    \end{align*}
    
    Reformulating the equation, we obtain
    $$S_m(W_{\vec{\bullet}}^{(i)};Y_{i+1})=\frac{m^n}{h^0(W_{m,\bullet}^{(n-1)})}\cdot \frac{1}{m^n}\sum_{\vec{a}\in\bb{N}^{n-1}}\frac{a_{i+1}}{m}\cdot \dim W^{(n-1)}_{m,\vec{a}}=\frac{m^n}{h^0(W_{m,\bullet}^{(n-1)})}\int_{\Delta}x_{i+1} d\rho_m,$$
    for all $0\leq i\leq n-2$. Here $\Delta$ is the Okounkov body of $W_{\vec{\bullet}}^{(n-1)}$ with respect to the flag $C\ni Q$, similar to Theorem \ref{read info from barycenter of convex body}.
    Therefore $$S(W_{\vec{\bullet}}^{(i)};Y_{i+1})=\lim_{m\rightarrow\infty}S_m(W_{\vec{\bullet}}^{(i)};Y_{i+1})=\frac{n!}{\operatorname{vol}(W_{\vec{\bullet}}^{(n-1)})}\int_{\Delta}x_{i+1}d\rho=\frac{1}{\operatorname{vol}(\Delta)}\int_{\Delta}x_{i+1}d\rho=b_{i+1},$$
    for all $0\leq i\leq n-2$. The last equality holds since the first $n-1$ coordinates of the barycenter of $\Delta$ coincide with those of $\bb{O}$, as we remarked in the proof of Theorem \ref{read info from barycenter of convex body}.

    Then we look at the last term. Since $Y_{n-1}$ is a smooth curve and $Z\cap Y_{n-1}$ is a collection of $k$ distinct closed points on $Y_{n-1}$, apply Lemma \ref{GeneralizedLocalDeltaOnCurve} and Theorem \ref{read info from barycenter of convex body} to get $$\delta_{Z\cap Y_{n-1}}(Y_{n-1};W_{\vec{\bullet}}^{(n-1)},\mathcal{F})\geq \frac{2k}{\deg c_1\left(W_{\vec{\bullet}}^{(n-1)}\right)+\lim_{m\rightarrow\infty}\deg F_{m}\left(W_{\vec{\bullet}}^{(n-1)} \right)}=\frac{2k}{2b_{n}}=\frac{k}{b_{n}}.$$This completes the proof.
\end{proof}
\begin{remark} 
    Recall the widely recognized result that, for a projective variety $X$ with at worst klt singularities, an ample $\bb{Q}$-line bundle $L$ on $X$, and an admissible flag $Y_{\bullet}:X=Y_0\supset Y_1\supset\dots \supset Y_{n-1}\ni Y_n$ with $n=\dim X$ and each $Y_i$ is a Cartier divisor on $Y_{i-1}$, we have $\delta_{Y_n}(X;L)\geq \min\left\{\frac{1}{b_1}, \dots, \frac{1}{b_n} \right\}.$ Here~$(b_1, \dots, b_n)$ is the barycenter of the Okounkov body of $L$ associated to $Y_{\bullet}$. This is the case $r=1$ in \autocite[Theorem 11.15]{Fujita_IMRN}. 

    Just as the convex body $\bb{O}(W_{\vec{\bullet}})$ constructed in (\ref{construction of O(W)}) can be viewed as an analogue of Okounkov body (see Proposition \ref{proposition: Okounkov body v.s O(W)}), Corollary \ref{abban-zhuang in terms of barycenter} can be viewed as an analogue of this classical result. 
    
    As we noted before, the first $n-1$ coordinates of the barycenters of the two convex bodies are the same. For the last coordinate, the one derived from $\bb{O}(W_{\vec{\bullet}})$ is clearly larger, however the numerator also increases from $1$ to $k$. Hence sometimes it may provide a better lower bound.  
\end{remark}

However, usually it is difficult to depict the convex body $\bb{O}(W_{\vec{\bullet}}^{(n-1)})$ and to compute, or to estimate, its barycenter, especially when $\dim X=n$ is high. We show that when the first few members of the admissible flag are proportional to the given ample line bundle $L$, we can reduce to the information of a convex body of lower dimension.

First we notice the following induction formula observed in \autocite[\S 3.1]{AZ1}:

Let $X$ be a projective variety with $Y\subset X$ an irreducible Cartier divisor. Let $V_{\vec{\bullet}}$ be a multi-graded linear series on $X$ and $W_{\vec{\bullet}}$ be its refinement by $Y$. Let $\mathcal{F}$ be a filtration indexed by some $\bb{R}^m$ on $W_{\vec{\bullet}}$. We also denote its lift by $\mathcal{F}$, which is an $\bb{R}^{m+1}$-indexed filtration on $V_{\vec{\bullet}}$. If $D$ is an $m$-basis type $\Q$-divisor of~$V_{\vec{\bullet}}$ compatible with $\mathcal{F}$, we have $$D=S_m(V_{\vec{\bullet}};Y)\cdot Y+\Gamma,$$ where $\Gamma\vert_Y$ is an $m$-basis type $\Q$-divisor of $W_{\vec{\bullet}}$ compatible with $\mathcal{F}$. Letting $m\rightarrow\infty$, we obtain 
\begin{equation}\label{induction formula 2}
    c_1(W_{\vec{\bullet}})=\left(c_1\left(V_{\vec{\bullet}}\right)-S\left(V_{\vec{\bullet}};Y\right)\cdot Y\right)\vert_{Y}.
\end{equation}

\begin{proposition}\label{reduce to lower dimension}
    Notation and assumptions as above. Assume that there exists some $0<\ell\leq n-2$ such that $Y_i\sim m_iL|_{Y_{i-1}}$ for all $1\leq i\leq \ell$ and some $m_i\in \bb{Z}_{>0}$. Set $Y=Y_{\ell}$, and $$Y^{\prime}_{\bullet} : Y=Y_\ell\supset Y_{\ell+1} \supset \dots \supset Y_{n-1}$$ be the induced admissible flag on $Y$. Let $U_{\bullet}^{(0)}:=\oplus_m H^0(Y, mL|_Y)$ be the complete linear series associated to $L|_Y$, and $U^{(i)}_{\vec{\bullet}}$ be the refinement of $U_{\bullet}^{(0)}$ by $Y^{\prime (i)}_{\bullet} : Y=Y_{\ell}\supset \dots \supset Y_{\ell+i}$. Let $B'=(b_1^{\prime}, \dots, b_{n-\ell}^{\prime})$ be the barycenter of $\bb{O}\left(U_{\vec{\bullet}}^{(n-\ell-1)}\right)$, then $$\delta_Z(X;L)\geq \min\left\{ (n+1)\cdot \min_{0\leq i< \ell}\left\{  m_{i+1} \right\},\quad\frac{n+1}{n-\ell+1}\cdot \min_{1\leq i\leq n-\ell-1}\left\{  \frac{1}{b_i^{\prime}} \right\},\quad\frac{(n+1)}{(n-\ell+1)}\cdot \frac{k}{b_{n-\ell}^{\prime}}   \right\}.$$
    As before, $Z\subset X$ is an irreducible closed subset such that the set-theoretic intersection $Z\cap Y_{n-1}$ consists of $k$ distinct closed points.
\end{proposition}
\begin{proof}
    Similarly as in the proof of Corollary \ref{abban-zhuang in terms of barycenter}, we have 
    $$\delta_Z(X;L)\geq \min\left\{ \min_{0\leq i< \ell}\left\{  \frac{A_{Y_i}(Y_{i+1})}{S(W_{\vec{\bullet}}^{(i)};Y_{i+1})} \right\},  \min_{\ell\leq i\leq n-2}\left\{  \frac{A_{Y_i}(Y_{i+1})}{S(W_{\vec{\bullet}}^{(i)};Y_{i+1})} \right\}, \delta_{Z\cap Y_{n-1}}(Y_{n-1};W_{\vec{\bullet}}^{(n-1)},\mathcal{F})   \right\}.$$

    Firstly, we look at the first $\ell$ terms. Note that $W_{\vec{\bullet}}^{(1)}$ is almost complete and $F(W_{\vec{\bullet}}^{(1)})=0$: by definition \begin{align*}
        W_{m,j}^{(1)}&=\operatorname{Im}\left(H^0(X,mL-jY)\rightarrow H^0(Y, (mL-jY)|_{Y}) \right)\\
        &=\operatorname{Im}\left(H^0(X,(m-jm_1)L)\rightarrow H^0(Y, ((m-jm_1)L)|_{Y})  \right),
    \end{align*} 
    and the map is surjective when $m-jm_1\gg 0$, since $H^1(X, (m-(j+1)m_1)L)$ vanishes. Therefore~$W_{m,j}^{(1)}=~\lvert(m-jm_1)L|_{Y}\rvert$ and $F_{m,j}=0$ when $m-jm_1\gg 0$, hence $F(W_{\vec{\bullet}}^{(1)})=0$. Inductively we know that $W_{\vec{\bullet}}^{(i)}$ is almost complete and $F(W_{\vec{\bullet}}^{(i)})=0$ for all $1\leq i\leq \ell$.
    
    We know that $c_1(V_{\bullet})=L$. Apply Lemma \ref{induction formula} and equation (\ref{induction formula 2}) iteratively, we can obtain $$c_1(W_{\vec{\bullet}}^{(i)})=\frac{n+1-i}{n+1}L|_{Y_{i}},\  S(W_{\vec{\bullet}}^{(i)}; Y_{i+1})=\frac{1}{m_{i+1}(n+1)},$$ for $0\leq i\leq \ell-1$, and $c_1(W_{\vec{\bullet}}^{(\ell)})=\frac{n+1-\ell}{n+1}L|_{Y_{\ell}}$.

    Since each $Y_{i+1}$ is a divisor on $Y_i$, the log discrepancy $A_{Y_i}(Y_{i+1})=1$ for all $i$. Therefore, when~$0\leq~i\leq~\ell-1$, we have $\frac{A_{Y_i}(Y_{i+1})}{S(W_{\vec{\bullet}}^{(i)};Y_{i+1})}=m_{i+1}(n+1)$.

    Next, we look at the middle $n-\ell-1$ terms.

    We denote $W_{\vec{\bullet}}=W_{\vec{\bullet}}^{(n-1)}, U_{\vec{\bullet}}=U_{\vec{\bullet}}^{(n-\ell-1)}$ for simplicity.
    
    It is straightforward to see that for any $\vec{a}=(a_0, a_1, \dots, a_{n-1})$,  when $a_0-\sum_{i=1}^\ell m_ia_i\gg 0$, we have $$W_{\vec{a}}=U_{a_0-\sum_{i=1}^\ell m_ia_i, a_{\ell+1},\dots, a_{n-1}}.$$ Conversely, we have $W_{\vec{a}}=0$ when $a_0-\sum_{i=1}^\ell m_ia_i<0$.
    
    For any $\gamma=(\gamma_1, \dots, \gamma_{n-1})\in\bb{Q}_{>0}^{n-1}$, denote $\gamma'=\sum_{i=1}^\ell m_i\gamma_i$ and $\tilde{\gamma}=(\gamma_{\ell+1}, \dots, \gamma_{n-1})$. Then $$\gamma\in\operatorname{Int}(\Delta_{\operatorname{Supp}(W_{\vec{\bullet}})})\cap \bb{Q}^{n-1}\iff \gamma^{\prime}\in (0,1) \text{\ and\ } \tilde{\gamma}/(1-\gamma^{\prime})\in\operatorname{Int}(\Delta_{\operatorname{Supp}(U_{\vec{\bullet}})})\cap \bb{Q}^{n-\ell-1},$$ and the fixed part $F_{m,m\gamma}(W_{\vec{\bullet}})=F_{m(1-\gamma^{\prime}),m\tilde{\gamma}}(U_{\vec{\bullet}})$.
  
    We define $f(\gamma)$ for all $\gamma\in\operatorname{Int}(\Delta_{\operatorname{Supp}(W_{\vec{\bullet}})})$ as in the proof of Lemma \ref{DegreeOfFixed}. Similarly, we define $f^U(\delta)$ for~$\delta\in \operatorname{Int}(\Delta_{\operatorname{Supp}(U_{\vec{\bullet}})})$. Then 
    \begin{align*}
        f(\gamma)&=\lim\limits_{m\rightarrow\infty}\frac{\deg F_{m,m\gamma}(W_{\vec{\bullet}})}{m}=\lim\limits_{m\rightarrow\infty}\frac{\deg F_{m(1-\gamma^{\prime}),m\tilde{\gamma}}(U_{\vec{\bullet}})}{m}=(1-\gamma^{\prime})f^U(\frac{\tilde{\gamma}}{1-\gamma^{\prime}}),\\
        \operatorname{vol}_{W_{\vec{\bullet}}}(\gamma)&=\lim\limits_{m\rightarrow\infty}\frac{\dim W_{m,m\gamma}}{m}=\lim\limits_{m\rightarrow\infty}\frac{\dim U_{m(1-\gamma^{\prime}),m\tilde{\gamma}}}{m}=(1-\gamma^{\prime})\operatorname{vol}_{U_{\vec{\bullet}}}(\frac{\tilde{\gamma}}{1-\gamma^{\prime}}).
    \end{align*}

    Therefore we have the surjective map 
    \begin{align*}
        p: \operatorname{Int}\left(\bb{O}(W_{\vec{\bullet}})\right)&\rightarrow \operatorname{Int}\left(\bb{O}\left(U_{\vec{\bullet}}\right)\right)\\
        (\gamma,z)&\mapsto \frac{1}{1-\gamma^{\prime}}(\tilde{\gamma},z)
    \end{align*}
    We remark that the map is well-defined. Indeed, if $(\gamma,z)\in \operatorname{Int}\left(\bb{O}(W_{\vec{\bullet}})\right)$, then $\gamma\in \operatorname{Int}\left(\Delta_{\operatorname{Supp}(W_{\vec{\bullet}})}\right)$ and~$z\in \left(f(\gamma), f(\gamma)+\operatorname{vol}_{W_{\vec{\bullet}}}(\gamma)\right)$. Then 
    \begin{equation*}
        \begin{gathered}
        \frac{\tilde{\gamma}}{1-\gamma^{\prime}}\in \operatorname{Int}\left(\Delta_{\operatorname{Supp}(U_{\vec{\bullet}})}\right),\\
        \frac{z}{1-\gamma^{\prime}}\in \left(\frac{f(\gamma)}{1-\gamma^{\prime}}, \frac{f(\gamma)+\operatorname{vol}_{W_{\vec{\bullet}}}(\gamma)}{1-\gamma^{\prime}}\right)=\left(f^U(\frac{\tilde{\gamma}}{1-\gamma^{\prime}}), f^U(\frac{\tilde{\gamma}}{1-\gamma^{\prime}})+\operatorname{vol}_{U_{\vec{\bullet}}}(\frac{\tilde{\gamma}}{1-\gamma^{\prime}})\right).
        \end{gathered}
    \end{equation*} 
    This implies that $\frac{1}{1-\gamma^{\prime}}(\tilde{\gamma},z)\in \operatorname{Int}\left(\bb{O}(U_{\vec{\bullet}})\right)$.

    For a point $(\delta, w)\in \bb{O}(U_{\vec{\bullet}})\subset \bb{R}^{n-\ell-1}\times \bb{R}$, the preimage is $$p^{-1}(\delta, w)=\left\{(\gamma_1, \dots, \gamma_{\ell}, (1-\gamma^{\prime})\delta, (1-\gamma^{\prime})w) \mid \gamma_i\in \bb{R}_{\geq 0}, i=1, \dots, \ell, \gamma^{\prime}=\sum_{i=1}^\ell m_i\gamma_i\in (0,1) \right\}.$$
    
    We see that $\bb{O}(U_{\vec{\bullet}})$ is a ``rescaling of the projection of $\bb{O}(W_{\vec{\bullet}})$ to the last $n-\ell$ coordinates".

    Note that $$\operatorname{vol}(W_{\vec{\bullet}})=\operatorname{vol}(L)=n!\operatorname{vol}(\bb{O}(W_{\vec{\bullet}})),\quad \operatorname{vol}(U_{\vec{\bullet}})=\operatorname{vol}(L|_{Y})=(\prod_{i=1}^\ell m_i) \operatorname{vol}(L)=(n-\ell)!\operatorname{vol}(\bb{O}(U_{\vec{\bullet}})).$$ Let $\vec{b}=(b_1,\dots, b_n)$ be the barycenter of $\bb{O}(W_{\vec{\bullet}})$, then we can relate $(b_{\ell+1}, \dots, b_{n})$ to $\vec{b}'=(b_1', \dots, b_{n-\ell}')$.

    Let $x=(\gamma, z)=(\gamma_1, \dots, \gamma_{n-1}, z)$ represent a point in $\bb{O}(W_{\vec{\bullet}})$, and $(\delta, w)=(\delta_1, \dots,\delta_{n-\ell-1},w)$ represent a point in $\bb{O}(U_{\vec{\bullet}})$. Then
    \begin{align*}
        b_{\ell+i}&=\frac{1}{\operatorname{vol}(\bb{O}(W_{\vec{\bullet}}))}\int_{\bb{O}(W_{\vec{\bullet}})}x_{\ell+i}dx=\frac{n!(\prod_{i=1}^{\ell} m_i)}{(n-\ell)!\operatorname{vol}(\bb{O}(U_{\vec{\bullet}}))}\int_{\begin{subarray}
        ((\gamma_1,\dots,\gamma_{\ell})\in\bb{R}^{\ell}_{\geq 0}\\0\leq \gamma^{\prime}\leq 1
        \end{subarray}}\int_{(1-\gamma^{\prime})\bb{O}(U_{\vec{\bullet}})}x_{\ell+i}dx\\
        &=\frac{n!(\prod_{i=1}^{\ell} m_i)}{(n-\ell)!\operatorname{vol}(\bb{O}(U_{\vec{\bullet}}))}\int_{\begin{subarray}
        ((\gamma_1,\dots,\gamma_{\ell})\in\bb{R}^{\ell}_{\geq 0}\\0\leq \gamma^{\prime}\leq 1
        \end{subarray}}\int_{\bb{O}(U_{\vec{\bullet}})}(1-\gamma^{\prime})\delta_{i}(1-\gamma^{\prime})^{n-\ell}d\delta_1\dots d\delta_{n-\ell-1}dw d\gamma_1\dots d\gamma_{\ell}\\
        &=\frac{n!(\prod_{i=1}^{\ell} m_i)}{(n-\ell)!}\int_{\begin{subarray}
        ((\gamma_1,\dots,\gamma_{\ell})\in\bb{R}^{\ell}_{\geq 0}\\0\leq \gamma^{\prime}\leq 1 \end{subarray}}(1-\gamma^{\prime})^{n-\ell+1}b_i'd\gamma_1\dots d\gamma_{\ell}\\
        &=\frac{n!(\prod_{i=1}^{\ell} m_i)}{(n-\ell)!}\cdot b_i'\cdot \frac{(n-\ell+1)!}{(n+1)!(\prod_{i=1}^{\ell} m_i)}=\frac{n-\ell+1}{n+1}b_i',
    \end{align*}
    for $1\leq i\leq n-\ell-1$. Therefore $$\frac{A_{Y_i}(Y_{i+1})}{S(W_{\vec{\bullet}}^{(i)};Y_{i+1})}=\frac{1}{b_{i+1}}=\frac{n+1}{n-\ell+1}\cdot \frac{1}{b_{i-\ell+1}'}$$ for all $\ell\leq i\leq n-2$.
    
    Now we look at the last term. We could expect that it differs from the last term $\frac{k}{b_n}$ in Corollary \ref{abban-zhuang in terms of barycenter} by a coefficient $\frac{n+1}{n-\ell+1}$, just like the middle terms. We give the proof below for rigor. Similar to the middle terms, we have 
    \begin{align*}
        &\lim\limits_{m\rightarrow\infty}\deg F_m(W_{\vec{\bullet}})=\frac{n!}{\operatorname{vol}(W_{\vec{\bullet}})}\int_{\Delta_{\operatorname{Supp}(W_{\vec{\bullet}})}} \operatorname{vol}_{W_{\vec{\bullet}}}(\gamma)f(\gamma)d\gamma\\
        =&\frac{n!}{\operatorname{vol}(L)}\int_{\begin{subarray}
        ((\gamma_1,\dots,\gamma_{\ell})\in\bb{R}^{\ell}_{\geq 0}\\0\leq \gamma^{\prime}\leq 1
        \end{subarray}}\int_{(1-\gamma^{\prime})\cdot\Delta_{\operatorname{Supp}(U_{\vec{\bullet}})}} (1-\gamma^{\prime})\operatorname{vol}_{U_{\vec{\bullet}}}(\frac{\tilde{\gamma}}{1-\gamma^{\prime}})\cdot (1-\gamma^{\prime})f^U(\frac{\tilde{\gamma}}{1-\gamma^{\prime}})d\gamma_{\ell+1}\dots d\gamma_{n-1}d\gamma_1\dots d\gamma_{\ell} \\
        =&\frac{n!\prod_{i=1}^{\ell} m_i}{(n-\ell)!}\cdot \lim\limits_{m\rightarrow \infty}\deg F_m(U_{\vec{\bullet}}) \cdot \int_{\begin{subarray}((\gamma_1,\dots,\gamma_{\ell})\in\bb{R}^{\ell}_{\geq 0}\\ 0\leq \gamma^{\prime}\leq 1\end{subarray}} (1-\gamma^{\prime})^{n-\ell+1}\cdot d\gamma_1\dots d\gamma_{\ell}=\frac{n-\ell+1}{n+1}\lim\limits_{m\rightarrow \infty}\deg F_m(U_{\vec{\bullet}}).
    \end{align*}
    Here the first and third equality use Lemma \ref{DegreeOfFixed}. 

    Moreover, following the notation in Theorem \ref{read info from barycenter of convex body}, we have $\deg c_1(W_{\vec{\bullet}})=\alpha+\sum_{i=1}^{n-1}\beta_i b_i$, here~$\alpha=~\deg L|_{Y_{n-1}}$, and $\beta_i=\deg (-Y_i)|_{Y_{n-1}}$. Especially $\beta_i=-m_i\alpha$ for $1\leq i\leq \ell$. We just showed that $b_i=\frac{1}{(n+1)m_i}$ for $1\leq i\leq \ell$, therefore $\deg c_1(W_{\vec{\bullet}})=(1-\frac{\ell}{n+1})\alpha+\sum_{i=\ell+1}^{n-1}\beta_i b_i$. Also~$\deg c_1(U_{\vec{\bullet}})=\alpha'+\sum_{i=1}^{n-\ell-1}\beta'_i b'_i$, with $\alpha'=\deg (L|_{Y_{\ell}})|_{Y_{n-1}}=\alpha$, and $\beta'_i=\deg (-Y_{\ell+i})|_{Y_{n-1}}=\beta_{\ell+i}$ for all $1\leq i\leq n-\ell-1$. Then we see that $$\deg c_1(W_{\vec{\bullet}})=(1-\frac{\ell}{n+1})\alpha'+\sum_{i=\ell+1}^{n-1}\beta_{i-\ell}'\cdot \frac{n-\ell+1}{n+1} b_{i-\ell}'=\frac{n-\ell+1}{n+1}\deg c_1(U_{\vec{\bullet}}).$$

    Then according to Theorem \ref{read info from barycenter of convex body}, 
    \begin{align*}
        b_n&=\frac{1}{2}\left(\lim_{m\rightarrow\infty}\deg F_m(W_{\vec{\bullet}})+\deg c_1(W_{\vec{\bullet}})\right)\\
        &=\frac{1}{2}\left(\frac{n-\ell+1}{n+1}\lim\limits_{m\rightarrow \infty}\deg F_m(U_{\vec{\bullet}})+\frac{n-\ell+1}{n+1}\deg c_1(U_{\vec{\bullet}})\right)=\frac{n-\ell+1}{n+1}b_{n-\ell}'.
    \end{align*}
    Combined with Corollary \ref{abban-zhuang in terms of barycenter}, the last term is bounded below by $$\delta_{Z\cap Y_{n-1}}(Y_{n-1};W_{\vec{\bullet}}^{(n-1)},\mathcal{F})\geq \frac{k}{b_n}=\frac{(n+1)k}{(n-\ell+1)b_{n-\ell}^{\prime}}.$$ This completes the proof.
\end{proof}
\begin{remark}
    The similar admissible flag, i.e. the member is proportional to the given line bundle, is also considered in \autocite[Lemma 4.1]{fujita2026kstabilityfanoweightedhypersurfaces}. The result is however not directly applicable to our setting, as~$Z\cap Y_{n-1}$ here is usually not irreducible.
\end{remark}

If we take $\ell=n-2$ in Proposition \ref{reduce to lower dimension}, then $\bb{O}(U_{\vec{\bullet}})$ is a 2-dimensional convex body. In such case we can give a practical estimate of its barycenter. We illustrate this in the next example.

\begin{example}\label{Fujitas Convex Body Barycenter}
    Let $S$ be a smooth projective surface, $C\subset S$ an irreducible smooth curve. Let $L$ be an ample line bundle on $S$, and $V_{\bullet}=\oplus H^0(S,mL)$ be the complete linear series associated to $L$. Let $W_{\bullet,\bullet}$ be the $\bb{N}^2$-graded linear series on $C$ associated to $L\vert_C$ and $(-C)\vert_C$, which is the refinement of $V_{\bullet}$ by $C$. Then $W_{\bullet,\bullet}$ has bounded support and contains an ample series.
    
    Set $T=T(L;C)$, then $\Delta_{\operatorname{Supp}(W_{\bullet,\bullet})}=[0,T]$. For $x\in (0,T)$, we can define $f(x)$ in the sense of Lemma~\ref{DegreeOfFixed}. Denote $\alpha=\deg L\vert_C>0, \beta=\deg (-C)\vert_C$, we know that $f(x)=\alpha+x\beta-\operatorname{vol}_{W_{\vec{\bullet}}}(x)$. For simplicity we denote $u(x)=\operatorname{vol}_{W_{\vec{\bullet}}}(x)$. We can extend both $f(x)$ and $u(x)$ continuously to $[0, T]$.
    
    Note $L\vert_C$ is ample on $C$, thus $(L-xC)\vert_C$ is also ample when $x$ is small. In this case $W_{m,mx}$ has no fixed part. That is, $f(x)=0$ when $x$ is small. We assume there exists $e\in(0,T)$ such that $f(x)=0$ for all $x\in [0,e]$. Also note that $u(x)>0$ when $x\in [0,T)$. Draw the convex body $\bb{O}(W_{\bullet,\bullet})$ defined in Definition \ref{definition: construction of the convex bodt}, it is the shaded area in Figure 1. The figure depicts the case $\beta>0$, but we can see that all arguments work for any $\beta\in \bb{R}$.

   \begin{figure}[htbp] 
        \centering 
        \begin{tikzpicture}[>=stealth, scale=0.8]
        \def\E{2}
        \def\T{6}
        \def\midX{4} 
    
        \fill[gray!20] (0,0) -- plot[domain=0:\E, samples=100] (\x, {2 + 0.5*\x}) -- (\E,0) -- cycle;

        \fill[gray!20] plot[domain=\E:\T, samples=100] (\x, {2 + 0.5*\x}) -- 
                   plot[domain=\T:\E, samples=100] (\x, {0.28125*(\x-\E)^2}) -- cycle;

        \draw[->, thick] (-1,0) -- (9.5,0) node[below right] {$x$};
        \draw[->, thick] (0,-1) -- (0,7);

        \draw[thick] (-0.5, {2 + 0.5*(-0.5)}) -- (9, {2 + 0.5*9}) node[pos=0.35, sloped, above] {$\alpha + \beta x$};
    
        \draw[thick, domain=\E:\T, samples=100] plot (\x, {0.28125*(\x-\E)^2});

        \draw[dashed] (\T, 0) -- (\T, 5);
    
        \draw[thick] (\T, 4.5) -- (\T, 5);

        \draw[dashed, thick] (\E, 0) -- (\T, 3.2);

        \node[font=\scriptsize, left] at (0,2.1) {$\alpha$};

        \fill (4.8444,2.2756) circle (1pt);
        \node[font=\scriptsize, below right, inner sep=1.5pt]
        at (4.2,2.7) {$(x_0,y_0)$};

        \fill (\T,3.2) circle (1pt);
        \node[font=\scriptsize, right, inner sep=1.5pt]
        at (\T,3.2) {$(T,\nu_0)$};

        \draw[thick] (\E,0.1) -- (\E,-0.1) node[below] {$e$};
        \draw[thick] (\T,0.1) -- (\T,-0.1) node[below] {$T$};

        \def\midYcurve{1.125} 
        \def\midYline{4}      

        \draw[<->] (\midX, \midYcurve) -- (\midX, \midYline) node[midway, left] {$u(x)$};
    
        \draw[<->] (\midX, 0) -- (\midX, \midYcurve) node[midway, right] {$f(x)$};

        \end{tikzpicture}

        \caption{}
        \label{fig:my_graph} 

    \end{figure}

    The volume of the shaded area is $V=\operatorname{vol}(L)/2$. 

    Let $(b_1, b_2)$ be its barycenter. We see from the following Proposition that $b_1$ and $b_2$ achieve the maximum possible value, $\tilde{b}_1$ and $\tilde{b}_2$, simultaneously, when $f(x)$ is linear on $[e,T]$. 

    Then for $Z=\{p_1, \dots, p_k\}\subset C$ a collection of $k$ distinct closed points on $C$, we can conclude that $$\delta_{Z}(C;W_{\vec{\bullet}},\mathcal{F})\geq \frac{k}{b_2}\geq \frac{k}{\tilde{b}_2}.$$ Here $\mathcal{F}$ is a filtration in the sense of Lemma \ref{GeneralizedLocalDeltaOnCurve}.

    It is also interesting to notice that $$\delta_C(S;L)=\delta_{C}(S;V_{\bullet})=\frac{1}{S(V_{\bullet};C)}=\frac{1}{b_1}\geq \frac{1}{\tilde{b}_1},$$ due to Lemma \ref{valuative criterion for multigraded} and the fact that the only prime divisor over $S$ whose center contains $C$ is $C$ itself. 
\end{example}

\begin{proposition}
    Let $V, \alpha, T$ be positive real numbers, let $0<e<T$ and $\beta\in\bb{R}$. Consider a convex body $\bb{O}\subset \bb{R}_{\geq 0}^2$ given by $$\bb{O}=\{(x,y)\mid 0\leq x\leq T, f(x)\leq y\leq \alpha+\beta T\},$$ where $f(x)$ is any convex function such that $f(x)=0$ on $[0,e]$ and $0\leq f(x)< \alpha+\beta x$ on $[0,T)$ such that $\operatorname{vol}(\bb{O})=V$.  Let $(b_1, b_2)$ be its barycenter. Then $b_1$ and $b_2$ achieve the maximum possible value simultaneously, when $f(x)$ is linear on $[e,T]$.
\end{proposition}
\begin{proof}
    The idea of the proof is inspired by \autocite[Proposition 6.1]{fujita2026kstabilityfanoweightedhypersurfaces}, but is much easier as we impose more assumptions.

    Firstly we note that such convex body must contain the quadrilateral $$\bb{K}:=\operatorname{Conv}\left( (0,0), (0,\alpha), (T,\alpha+\beta T), (e,0)\right)$$whose volume is at most $V$. Then there exists a unique $\nu_0\in [0,\alpha+\beta T]$ such that the convex body~$\bb{O}':=~\operatorname{Conv}\left( (0,0), (0,\alpha), (T,\alpha+\beta T), (T, \nu_0), (e,0) \right)$ has volume $V$. Note that $(T,\alpha+\beta T)$ and~$(T, \nu_0)$ could be the same point, in such case $\bb{K}=\bb{O}'$.
    
    Let the barycenter of $\bb{O}'$ be $(\tilde{b}_1, \tilde{b}_2)$. See the dashed line in Figure 1. 

    For another $f(x)$, if any, such that it determines an eligible convex body $\bb{O}$ of volume $V$, it must intersect the dashed line at exactly one point for $x\in (e,T)$: otherwise over $(e,T)\subset \bb{R}$ it is either always above, or under the dashed line, then the volumes cannot coincide.  Assume they intersect at point $(x_0, y_0)$. 

    Set $A_1=\bb{O}\cap \bb{O}', A_2=\bb{O}'\backslash A_1$, and $A_3=\bb{O}\backslash A_1$.
    
    Then $\operatorname{vol}(A_2)=\operatorname{vol}(A_3)=V-\operatorname{vol}(A_1)$. We see that for any point $(x,y)\in A_2, x\geq x_0$ and $y\geq y_0$. While for any $(x,y)\in A_3, x\leq x_0$ and $y\leq y_0$. 

    Therefore, 
    \begin{align*}
        \tilde{b}_1&=\frac{1}{V}\int_{A_1\sqcup A_2}xdxdy=\frac{1}{V}\left(\int_{A_1}xdxdy+ \int_{A_2}xdxdy\right)\geq \frac{1}{V}\left(\int_{A_1}xdxdy+ \int_{A_2}x_0dxdy\right)\\
        &=\frac{1}{V}\left(\int_{A_1}xdxdy+ x_0 \operatorname{vol}(A_2)\right)=\frac{1}{V}\left(\int_{A_1}xdxdy+ x_0 \operatorname{vol}(A_3)\right)\\
        &=\frac{1}{V}\left(\int_{A_1}xdxdy+ \int_{A_3}x_0dxdy\right)\geq \frac{1}{V}\left(\int_{A_1}xdxdy+ \int_{A_3}xdxdy\right)=b_1
    \end{align*}
    With exactly the same argument, we can show that $\tilde{b}_2\geq b_2$. Therefore, we show that $b_1$ and $b_2$ achieve the maximum possible value simultaneously, and the maximum values are $\tilde{b}_1$ and $\tilde{b}_2$, respectively.
\end{proof}

\begin{remark}\label{remark after example 2D}
    In practice, it may be difficult to determine the pseudo-effective threshold $T=T(L;C)$. Actually, our construction provides a simple estimate for $T$. In the discussion of Example \ref{Fujitas Convex Body Barycenter}, the convex body must contain the quadrilateral $\bb{K}$ of area$=\frac{1}{2}\left(\alpha T+e\left(\alpha+\beta T\right)\right)\leq V$. Hence $T\leq \tau_0$ where $\tau_0$ is defined by $\frac{1}{2}\left(\alpha \tau_0+e\left(\alpha+\beta \tau_0\right)\right)= V$. In the situation $T=\tau_0$, the maximum possible coordinates of the barycenter are given by: 
    \begin{align}\label{naive maximum barycenter in 2D}
        (\tilde{b}_1,\tilde{b}_2)=\left(\frac{e+T}{3}-\frac{\alpha eT}{6V}, \frac{\alpha+\beta T}{3}+\frac{\alpha^2T}{6V} \right).
    \end{align}
    If a better estimate $T\leq \tau \leq \tau_0$ is available for a specific surface and curve, then taking $T=\tau$ yields a better upper bound for the barycenter coordinates. It is not hard to see that the maximum coordinates increase as $T$ increases, the resulting values are greater than or equal to those corresponding to the actual~$T=T(L;C)$. The estimate for local stability threshold we then obtain is clearly not optimal, although it can still be useful in some cases. We will see in the proof of Theorem A, particularly in Lemma \ref{contain line}, that even the naive estimate $\tau_0$ will do the job. 
\end{remark}

\section{Proof of Theorem A}\label{proof of theorem A}
Now we look at Theorem A. To apply the theories we built in Section \ref{Stability Threshold Along Several Points}, we take the first few members of the flag from the class of restrictions of $-K_X$, as in Proposition \ref{reduce to lower dimension}. Note the anticanonical class in the setting of Theorem A is exactly the class of hyperplane sections. With the following lemma we see that the flag can be chosen smooth until the dimension is quite small. 

\begin{lemma}\label{smooth}
    Let $X$ be a smooth complete intersection of codimension 2 in $\mathbb{P}^{n+2}$ and $n>1$. Assume~$X$ is not contained in any hyperplane. Let $P$ and $Q$ be two distinct closed points on X, then a general hyperplane section of $X$ passing through P and Q is smooth unless $n=2$ and $X$ contains the line joining~$P$ and $Q$.
\end{lemma}
\begin{proof}
    This is essentially the same as the proof of \autocite[Lemma 4.14]{AZ1}, but here the codimension is 2. 

    Let $L$ be the line joining $P$ and $Q$, let $V\subset \mathbb{P}\left(H^0(X,\mathcal{O}_X(1))\right)$ consist of hyperplane sections containing~$L\cap X$. Since $\dim H^0(X,\mathcal{O}_X(1))=n+3$, $\dim V= n$. Moreover, by Bertini's theorem, a general member in $V$ is smooth away from $L\cap X$. 

    Note that the hyperplane sections passing through any closed $x\in X$ and being not regular at $x$ gives a 1-dimensional family in $\mathbb{P}\left(H^0(X,\mathcal{O}_X(1))\right)$: Fix $f_0\in U=H^0(X,\mathcal{O}_X(1))$ such that $(f_0=0)$ defines a hyperplane section in $X$ avoiding $x$. Then we can define the map 
    $$\varphi_x: U \rightarrow \mathcal{O}_{X,x}/\mathfrak{m}_x^2$$ by sending each $f\in U$ to the stalk of $(f/f_0)$ at $x$. It is well defined since the pole of $(f/f_0)$ avoids $x$. Note that $f\in V$ gives a hyperplane section passing through $x$ if and only if $\varphi_x(f)\in \mathfrak{m}_x/\mathfrak{m}_x^2$, and it is not regular at $x$ if and only if $\varphi_x(f)=0$. Since $x$ is a smooth point, we know that $\mathcal{O}_{X,x}/\mathfrak{m}_x^2$ is generated by constant and linear terms, so $\dim \mathcal{O}_{X,x}/\mathfrak{m}_x^2=n+1$. Moreover, $\varphi_x$ is surjective since the very ample line bundle $\mathcal{O}(1)$ separates tangent vectors at $x$. Therefore $\dim (\mathrm{Ker}(\varphi_x))=\dim U - (n+1)=2$ and such divisors in $\bb{P}(U)$ are 1-dimensional.

    If $|L \cap X|<\infty$, a general member in $V$ is smooth at every point in $L\cap X$, as $\dim V=n\geq 2$. If~$L\subset X$, all hyperplane sections which are not smooth on $L$ form a subvariety of dimension at most 2. So when $\dim V=n>2$, a general member in $V$ is smooth.
\end{proof}

To verify the first condition of Lemma \ref{criterionInAZ}, we need to determine $\delta_C(X)$ for all irreducible curve~$C\subset X$. We firstly treat an easy case, guided by Lemma \ref{smooth}.

\begin{lemma}\label{not contain line}
    Let $n\geq 2$, let $X$ be a smooth member in $\mathcal{X}_{2,n}$, and $C\subset X$ be an irreducible curve. If~$\deg (C)>1$, and there exist two distinct points $P, Q\in C$ such that the line $\overline{PQ}$ joining $P$ and $Q$ is not contained in $X$, then $\delta_C(X)\geq \frac{n+1}{n}$.
\end{lemma}
\begin{proof}
    Choose the flag 
    $$Y_{\bullet}\ :\ X=Y_0\supset Y_1\supset \dots \supset Y_{n-1}$$
    such that $Y_1$ is a general smooth hyperplane section intersecting $C$ at finitely many points including $P$ and $Q$, and $Y_i$ is a smooth general hyperplane section in $Y_{i-1}$ passing through $P$, $Q$ for $1\leq i\leq n-1$. 

    Denote $H=-K_X$. Let $V_{\bullet}$ be the complete linear series associated to $-K_X$, and $W_{\vec{\bullet}}^{(j)}$ be the refinement of $V_{\bullet}$ by $Y_{\bullet}^{(j)}\colon \ X=Y_0\supset Y_1\supset \dots \supset Y_j$. Let $\mathcal{F}$ be the filtration on $W_{\vec{\bullet}}^{(n-1)}$ induced by a very general point $R\in Y_{n-1}\backslash C$, in the sense of Lemma \ref{GeneralizedLocalDeltaOnCurve}. We also use $\mathcal{F}$ for its lift to each $W_{\vec{\bullet}}^{(i)}$. Then following the proof of Proposition \ref{reduce to lower dimension} for the first few terms,
       
    we can compute:
    $$c_1(W_{\vec{\bullet}}^{(j)})=\frac{n+1-j}{n+1}H|_{Y_j},\quad S(W_{\vec{\bullet}}^{(j)};Y_{j+1})=\frac{1}{n+1}$$
    for all $0\leq j\leq n-2$.

    Then $$\left. c_1(W_{\vec{\bullet}}^{(n-1)})=\left(c_1(W_{\vec{\bullet}}^{(n-2)})-S(W_{\vec{\bullet}}^{(n-2)};Y_{n-1})Y_{n-1}\right)\right|_{Y_{n-1}}=\frac{2}{n+1}H|_{Y_{n-1}},\quad \deg c_1(W_{\vec{\bullet}}^{(n-1)})=\frac{4n}{n+1}.$$

    Note in this case $W_{\vec{\bullet}}$ has no fixed part, then using Lemma \ref{GeneralizedLocalDeltaOnCurve} we obtain $$\delta_{C\cap Y_{n-1}}(Y_{n-1};W_{\vec{\bullet}}^{(n-1)},\mathcal{F})\geq \frac{2\cdot 2}{\deg c_1(W_{\vec{\bullet}}^{(n-1)})}\geq \frac{n+1}{n},$$ as $C\cap Y_{n-1}$ contians at least 2 points. 

    We draw the conclusion by applying Theorem \ref{AZ Admissible Flag}:
    $$\delta_C(X;-K_X)\geq \min\left\{\min\limits_{0\leq i\leq n-2}\{n+1\}, \delta_{C\cap Y_{n-1}}(Y_{n-1};W_{\vec{\bullet}}^{(n-1)},\mathcal{F})\right\}\geq \frac{n+1}{n}.$$
\end{proof}

Now we consider the remaining case, in which any line joining two distinct points on the curve $C$ is contained in $X$. This further splits into two cases: if $C$ is a line or not. In both cases, we can still pick a similar smooth flag as in Lemma \ref{not contain line}, but only up to dimension 2. If $C$ is a line, then every member of the flag contains $C$ and we can simply take $C$ as the next member. This case is discussed in Lemma \ref{line}. We see in the following lemma that the most technical case is when $C$ itself is not a line.

\begin{lemma}\label{contain line}
    Let $n\geq 3$, let $X$ be a smooth member in $\mathcal{X}_{2,n}$, and $C\subset X$ be an irreducible curve. Assume that $\deg (C)>1$, and for every two distinct points $P, Q\in C$, the line $\overline{PQ}$ joining $P$ and $Q$ is contained in $X$. Then $\delta_C(X)\geq \frac{n+1}{n}$.
\end{lemma}
\begin{proof}
    Fix two distinct points $P$, $Q$ on $C$ and consider the following flag: 
    $$Y_{\bullet}\ :\ X=Y_0\supset Y_1\supset \dots \supset Y_{n-1}$$
    such that every $Y_i$ is a smooth hyperplane section in $Y_{i-1}$ and contains $\overline{PQ}$ for all $1\leq i\leq n-2$. We take~$Y_{n-1}$ to be the line $\overline{PQ}$.

    We use the same notation as in Lemma \ref{not contain line}, and by the same argument we know that $W_{\vec{\bullet}}^{(j)}$ is almost complete and $F(W_{\vec{\bullet}}^{(j)})=0$ for all $1\leq j\leq n-2$. In this case, let $\mathcal{F}$ be a filtration on $W_{\vec{\bullet}}^{(n-1)}$ induced by $\mathrm{ord}_R$ for a very general point $R\in Y_{n-1}$ outside $C$. Similarly we also use $\mathcal{F}$ for its lift to each $W_{\vec{\bullet}}^{(i)}$. The very generality condition follows Lemma \ref{GeneralizedLocalDeltaOnCurve}.

    Thanks to Proposition \ref{reduce to lower dimension}, we can instead think about the complete linear series $$V_{\bullet}=\bigoplus_{m\geq 0}H^0(Y_{n-2},-mK_X|_{Y_{n-2}})=\bigoplus_{m\geq 0}H^0(Y_{n-2},\mathcal{O}_{Y_{n-2}}(m))$$ and its refinement $W_{\bullet, \bullet}$ by $Y_{n-1}$.

    Now we take a closer look at the geometry of the surface. Let $S=Y_{n-2}\subset \mathbb{P}^4$, it is a smooth complete intersection of hypersurfaces of degree 2 and $n$. Let $L=\overline{PQ}\subset S$ be the line passing through $P$ and $Q$, and let $H$ be the hyperplane class on $S$. Then we can easily compute that 
    $$(H^2)=2n, (H\cdot L)=1,\text{and\ } (L^2)=1-n.$$

    Note we claim that $H-L$ is nef. For any irreducible curve $B\neq L$ in $S$, choose a hyperplane in $\bb{P}^4$ containing $L$ but not $B$. Writing its restriction to $S$ as $L+E$, then we obtain $\left(\left(H-L\right)\cdot B\right)=(E\cdot B)\geq 0$. When $B=L$, by direct computation $\left(\left(H-L\right)\cdot L\right)=n> 0$.

    Also note that the pseudo-effective threshold $T=T(H;L)\leq \frac{2n-1}{n}$. Actually, for any $x>\frac{2n-1}{n}$, if there exists an effective $\bb{R}$-divisor $D\sim_{\bb{R}} H-xL$, then $\left(D\cdot (H-L)\right)\geq 0$ as $H-L$ is nef. But~$\left((H-xL)\cdot (H-L)\right)=2n-1-xn<0.$

    Following the notation in Example \ref{Fujitas Convex Body Barycenter}, we can consider a two-dimensional convex body $\bb{O}=\bb{O}(W_{\bullet, \bullet})$ which is bounded above by the linear function $y=1+(n-1)x$, and we take $e=1, T=\frac{2n-1}{n}$. The volume of $\bb{O}$ is $V=\operatorname{vol}(H)/2=n.$

    By easy computation, we see $T=\frac{2n-1}{n}$ is exactly the estimation $\tau_0$ in Remark \ref{remark after example 2D}, then using equation (\ref{naive maximum barycenter in 2D}) we have $$(\tilde{b}_1, \tilde{b}_2)=\left(\frac{6n^2-4n+1}{6n^2},\frac{4n^3-4n^2+4n-1}{6n^2}\right).$$

    Now we can apply Proposition \ref{reduce to lower dimension} with $\ell=n-2$, and $b_1^{\prime}\leq \tilde{b}_1<1, b_2^{\prime}\leq \tilde{b}_2<\frac{2n}{3}$. Then $$\delta_C(X;-K_X)\geq \min\left\{ n+1, \frac{n+1}{3}\cdot \frac{1}{\tilde{b}_1},\frac{n+1}{3}\cdot \frac{2}{\tilde{b}_2}\right\}=\frac{n+1}{3}\cdot \frac{2}{\tilde{b}_2}\geq \frac{n+1}{n}.$$
\end{proof}
 
\begin{lemma}\label{line}
    Let $n\geq 3$, let X be a smooth member in $\mathcal{X}_{2,n}$, and $C\subset X$ be a curve with $\deg (C)=1$, then~$\delta_C(X)\geq \frac{n+1}{n}$.
\end{lemma}
\begin{proof}
    When $C$ is a line, we consider the following flag:$$Y_{\bullet}\ :\ X=Y_0\supset Y_1\supset \dots \supset Y_{n-2}$$
    such that every $Y_i$ is a smooth hyperplane section in $Y_{i-1}$ containing $C$ for all $1\leq i\leq n-2$.

    Taking trivial filtration, Theorem \ref{AZ Admissible Flag} gives
    \begin{equation*}\label{inequality2}
        \delta_C(X;-K_X)\geq \min\left\{\min\limits_{0\leq i\leq n-3}\{n+1\}, \delta_{C\cap Y_{n-2}}(Y_{n-2};W_{\vec{\bullet}}^{(n-2)})\right\}=\min\left\{n+1, \delta_{C}(Y_{n-2};W_{\vec{\bullet}}^{(n-2)})\right\},
    \end{equation*}
    as $C\cap Y_{n-2}=C$.
    We only need to compute the last term. Since $C$ is a smooth divisor on the surface~$Y_{n-2}$, the only prime divisor over $Y_{n-2}$ with center containing $C$ is $C$ itself. According to Lemma \ref{valuative criterion for multigraded}, the last term is computed as $$\delta_{C}(Y_{n-2};W_{\vec{\bullet}}^{(n-2)})=\frac{1}{S(W_{\vec{\bullet}}^{(n-2)};C)}.$$ 

    Since $W_{\vec{\bullet}}^{(n-2)}$ satisfies assumptions in Lemma \ref{induction formula}, and $c_1(W_{\vec{\bullet}}^{(n-2)})=\frac{3}{n+1}H$, we have $$S(W_{\vec{\bullet}}^{(n-2)};C)=\frac{3}{n+1}S(H;C).$$

    With Example \ref{Fujitas Convex Body Barycenter} and computation in the proof of Lemma \ref{contain line}, we know $S(H;C)=b_1\leq \tilde{b}_1<1$. Therefore $$\delta_C(X;-K_X)\geq \min \{n+1, \frac{n+1}{3b_1}\}=\frac{n+1}{3b_1}\geq \frac{n+1}{3}.$$
\end{proof}

Finally we justify the second condition in Lemma \ref{criterionInAZ}.

\begin{lemma}\label{beta}
    Let $n\geq 3$, let $X$ be a smooth member in $\mathcal{X}_{2,n}$, and $x\in X$ an arbitrary closed point. Let $E_x$ be the exceptional divisor of the blowup at $x$. Then $\beta(E_x)>0$ when $n\geq 3$.
\end{lemma}
\begin{proof}
    Let $H$ be the hyperplane class, then $T(H;E_x)=\sup\{\mathrm{mult}_x(D) \mid 0\leq D\sim_{\mathbb{Q}} H\}$. For any curve~$C\subset~X$ being smooth at $x$, we know $C.D\geq \mathrm{mult}_x(C)\cdot \mathrm{mult}_x(D)$. Choose $C=H_1\cap~\ldots\cap H_{n-1}$ a complete intersection of hyperplanes being smooth at $x$ and not contained in $\operatorname{Supp}D$. Then~$C.D=~(H^n)=2n$, so $T(H; E_x)\leq 2n$ for any $D\sim_{\mathbb{Q}} H$. Moreover, since $H$ is very ample, we know the movable threshold~$\eta(H;E_x)\leq \sqrt{(H^n)}=\sqrt{2n}$ from Lemma \ref{lemma on movable threshold}.

    Since $X$ is $\Q$-factorial and Picard rank $\rho(X)=1$, we can use the inequality given in \autocite[Lemma 4.2]{ZhuangSupperrigidity} to obtain
    $$S(-K_X;E_x)=S(H;E_x)\leq \frac{1}{n+1}T(H;E_x)+\frac{n-1}{n+1}\eta(H;E_x)\leq \frac{2n}{n+1}+\frac{n-1}{n+1}\cdot \sqrt{2n}.$$
    
    When $n=3, 4$ or $5$, can check that $S(-K_X;E_x)<n$ manually.

    When $n>5$, note $S(-K_X;E_x)< 2+\sqrt{2n}$. Let $h(t)=2+\sqrt{2t}-t$, since $h'(t)<0$ when $t>1$ and~$h(6)<0$, we can conclude $h(n)<0$ for all $n>5$.

    Since $x\in X$ is a smooth point, we have $A_X(E_x)=n$, and $\beta(E_x)=A_X(E_x)-S(-K_X;E_x)>0$. 
\end{proof}

\begin{proof}[Proof of Theorem \ref{main theorem}]
    When $n=2$, a smooth member in $\mathcal{X}_{2,2}$ is a smooth del Pezzo surface of degree 4. \cite{KstabilityOfDelPezzo} showed that it is K-polystable. Moreover, since $|\mathrm{Aut}(X)|<\infty$, from Theorem \ref{finite automorphism} we know it is K-stable.
    
    When $n\geq 3$, Lemma \ref{criterionInAZ} is satisfied: the first condition is justified in Lemmas \ref{not contain line}, \ref{contain line} and \ref{line}, and the second condition is justified in Lemma \ref{beta}.
\end{proof}

\medskip
\printbibliography

@article{AZ1,
   author = {Abban, Hamid and Zhuang, Ziquan},
   title = {K-stability of Fano varieties via admissible flags},
   journal = {Forum of Mathematics, Pi},
   volume = {10},
   ISSN = {2050-5086},
   DOI = {10.1017/fmp.2022.11},
   year = {2022},
   type = {Journal Article}
}

@article{AZ2,
   author = {Abban, Hamid and Zhuang, Ziquan},
   title = {Seshadri Constants and K-Stability of Fano Manifolds},
   journal = {Duke Mathematical Journal},
   volume = {172},
   number = {6},
   pages = {1109-1144},
   ISSN = {0012-7094},
   DOI = {10.1215/00127094-2022-0026},
   url = {<Go to ISI>://WOS:000978265000002},
   year = {2023},
   type = {Journal Article}
}

@article{ZhuangSupperrigidity,
   author = {Zhuang, Ziquan},
   title = {Birational superrigidity and K-stability of Fano complete intersections of index 1},
   journal = {Duke mathematical journal},
   volume = {169},
   number = {12},
   pages = {2205-2229},
   ISSN = {0012-7094},
   DOI = {10.1215/00127094-2020-0010},
   year = {2020},
   type = {Journal Article}
}

@article{Berman,
   author = {Berman, Robert J.},
   title = {K-polystability of $\mathbb{Q}$-Fano varieties admitting Kähler-Einstein metrics},
   journal = {Inventiones Mathematicae},
   volume = {203},
   number = {3},
   pages = {973-1025},
   DOI = {10.1007/s00222-015-0607-7},
   year = {2016},
   type = {Journal Article}
}

@article{BJThresholdsValuationandK-stability,
   author = {Blum, Harold and Jonsson, Mattias},
   title = {Thresholds, valuations, and K-stability},
   journal = {Advances in mathematics},
   volume = {365},
   %pages = {107062},
   ISSN = {0001-8708},
   DOI = {10.1016/j.aim.2020.107062},
   year = {2020},
   type = {Journal Article}
}

@article{ChenDnoaldosnSunI,
   author = {Chen, Xiuxiong and Donaldson, Simon and Sun, Song},
   title = {K$\ddot{a}$hler-Einstein metrics on Fano manifolds. I: Approximation of metrics with cone singularities},
   journal = {Journal of the American Mathematical Society},
   volume = {28},
   number = {1},
   pages = {183-197},
   ISSN = {0894-0347},
   DOI = {10.1090/S0894-0347-2014-00799-2},
   year = {2015},
   type = {Journal Article}
}

@article{ChenDnoaldosnSunII,
   author = {Chen, Xiuxiong and Donaldson, Simon and Sun, Song},
   title = {K$\ddot{a}$hler-Einstein metrics on Fano manifolds. II: limits with cone angle less than 2$\pi$},
   journal = {Journal of the American Mathematical Society},
   volume = {28},
   number = {1},
   pages = {199-234},
   ISSN = {0894-0347},
   DOI = {10.1090/S0894-0347-2014-00800-6},
   year = {2015},
   type = {Journal Article}
}

@article{ChenDnoaldosnSunIII,
   author = {Chen, Xiuxiong and Donaldson, Simon and Sun, Song},
   title = {K$\ddot{a}$hler-Einstein metrics on Fano manifolds. III: limits as cone angle approaches 2$\pi$ and completion of the main proof},
   journal = {Journal of the American Mathematical Society},
   volume = {28},
   number = {1},
   pages = {235-278},
   ISSN = {0894-0347},
   DOI = {10.1090/S0894-0347-2014-00801-8},
   year = {2015},
   type = {Journal Article}
}

@book{Xubook, 
place={Cambridge}, 
series={New Mathematical Monographs}, 
title={K-stability of {F}ano Varieties}, 
publisher={Cambridge University Press}, 
author={Xu, Chenyang}, 
year={2025}, 
collection={New Mathematical Monographs}
}

@article{TianK-stability,
   author = {Tian, Gang},
   title = {{K}{\"a}hler-{E}instein metrics with positive scalar curvature},
   journal = {Inventiones mathematicae},
   volume = {130},
   number = {1},
   pages = {1-37},
   ISSN = {0020-9910},
   DOI = {10.1007/s002220050176},
   year = {1997},
   type = {Journal Article}
}

@article{LXZ,
   author = {Liu, Yuchen and Xu, Chenyang and Zhuang, Ziquan},
   title = {Finite generation for valuations computing stability thresholds and applications to {K}-stability},
   journal = {Annals of Mathematics},
   volume = {196},
   number = {2},
   pages = {507-566},
   DOI = {10.4007/annals.2022.196.2.2},
   url = {https://doi.org/10.4007/annals.2022.196.2.2},
   year = {2022},
   type = {Journal Article}
}

@article{ReductiveAutomorphism,
   author = {Alper, Jarod and Blum, Harold and Halpern-Leistner, Daniel and Xu, Chenyang},
   title = {Reductivity of the automorphism group of {K}-polystable {F}ano varieties},
   journal = {Inventiones mathematicae},
   volume = {222},
   number = {3},
   pages = {995-1032},
   ISSN = {0020-9910},
   DOI = {10.1007/s00222-020-00987-2},
   year = {2020},
   type = {Journal Article}
}

@article{FiniteAutomorphism,
   author = {Blum, Harold and Xu, Chenyang},
   title = {Uniqueness of {K}-polystable degenerations of {F}ano varieties},
   journal = {Annals of mathematics},
   volume = {190},
   number = {2},
   pages = {609-656},
   ISSN = {0003-486X},
   DOI = {10.4007/annals.2019.190.2.4},
   year = {2019},
   type = {Journal Article}
}

@article{KstabilityOfDelPezzo,
   author = {Tian, Gang and Yau, Shing-Tung},
   title = {{K}{\"a}hler-{E}instein metrics on complex surfaces with $c_1>0$},
   journal = {Communications in Mathematical Physics},
   volume = {112},
   number = {1},
   pages = {175-203},
   ISSN = {1432-0916},
   DOI = {10.1007/BF01217685},
   url = {https://doi.org/10.1007/BF01217685},
   year = {1987},
   type = {Journal Article}
}

@article{Fujita_2019, 
title={K-stability of {F}ano manifolds with not small alpha invariants}, 
volume={18}, 
DOI={10.1017/S1474748017000111}, 
number={3}, 
journal={Journal of the Institute of Mathematics of Jussieu}, 
author={Fujita, Kento}, 
year={2019}, 
pages={519–530}}

@article{ZhuangOptimalDestabilizing,
   author = {Zhuang, Ziquan},
   title = {Optimal destabilizing centers and equivariant {K}-stability},
   journal = {Inventiones mathematicae},
   volume = {226},
   number = {1},
   pages = {195-223},
   ISSN = {0020-9910},
   DOI = {10.1007/s00222-021-01046-0},
   year = {2021},
   type = {Journal Article}
}

@article {FujitaOdakaLowerAlpha,
    AUTHOR = {Fujita, Kento and Odaka, Yuji},
     TITLE = {On the {K}-stability of {F}ano varieties and anticanonical
              divisors},
   JOURNAL = {Tohoku Mathematical Journal},
  FJOURNAL = {The Tohoku Mathematical Journal. Second Series},
    VOLUME = {70},
      YEAR = {2018},
    NUMBER = {4},
     PAGES = {511--521},
      ISSN = {0040-8735,2186-585X},
   MRCLASS = {14J45 (32Q26 53C55)},
  MRNUMBER = {3896135},
MRREVIEWER = {Andreas\ H\"oring},
       DOI = {10.2748/tmj/1546570823},
       URL = {https://doi.org/10.2748/tmj/1546570823}
}

@article{lazarsfeld2009convex,
  author  = {Lazarsfeld, Robert and Musta{\c{t}}{\u{a}}, Mircea},
  title   = {Convex Bodies Associated to Linear Series},
  journal = {Annales scientifiques de l'\'Ecole Normale Sup\'erieure},
  volume  = {42},
  number  = {5},
  pages   = {783--835},
  year    = {2009},
  doi     = {10.24033/asens.2109}
}

@article{szekelyhidi2015filtrations,
  title={Filtrations and test-configurations: With an appendix by Sebastien Boucksom},
  author={Sz{\'e}kelyhidi, G{\'a}bor},
  journal={Mathematische Annalen},
  volume={362},
  number={1},
  pages={451--484},
  year={2015},
  publisher={Springer}
}

@article{fujita2026kstabilityfanoweightedhypersurfaces,
      title={K-stability of Fano weighted hypersurfaces via plt flags and convex geometry}, 
      author={Livia Campo and Kento Fujita and Taro Sano and Luca Tasin},
      year={2026},
      journal={arXiv},
      number={2601.02974},
      archivePrefix={arXiv},
      primaryClass={math.AG},
      url={https://arxiv.org/abs/2601.02974}, 
}

@article{boucksom2014corps,
  title={Corps d'Okounkov (d'apr{\`e}s Okounkov, Lazarsfeld-Musta\c{t}\u{a} et Kaveh-Khovanskii)},
  author={Boucksom, S{\'e}bastien},
  journal={S{\'e}minaire Bourbaki. Ast{\'e}risque},
  number={361},
  year={2014},
  pages={1--41},
}

@article{Fujita_IMRN,
  title={On {K}-stability for {F}ano threefolds of rank 3 and degree 28},
  author={Fujita, Kento},
  journal={International Mathematics Research Notices},
  number={15},
  pages={12601--12784},
  year={2023},
  publisher={Oxford University Press}
}

@article{trusiani2026YTD,
      title={A solution to the {Y}au-{T}ian-{D}onaldson conjecture through special {F}ujita approximations}, 
      author={Antonio Trusiani},
      year={2026},
      journal={arXiv},
      number={2605.30063},
      archivePrefix={arXiv},
      primaryClass={math.AG},
      url={https://arxiv.org/abs/2605.30063}, 
}

@article{liu2022CubicFourfold,
  title={{K}-stability of cubic fourfolds},
  author={Liu, Yuchen},
  journal={Journal f{\"u}r die reine und angewandte Mathematik},
  %volume={2022},
  volume={786},
  pages={55--77},
  year={2022},
  publisher={De Gruyter}
}

@article{Ghigi2004symmetries,
  title={Symmetries, quotients and {K}{\"a}hler-{E}instein metrics},
  author={Arezzo, Claudio and Ghigi, Alessandro and Pirola, Gian Pietro},
  journal={Journal f{\"u}r die reine und angewandte Mathematik},
  year={2006},
  volume={591},
  pages={177--200}
}

@article{JonssonMustata,
     author = {Jonsson, Mattias and Musta\c{t}\u{a}, Mircea},
     title = {Valuations and asymptotic invariants for sequences of ideals},
     journal = {Annales de l'Institut Fourier},
     pages = {2145--2209},
     year = {2012},
     %publisher = {Association des Annales de l{\textquoteright}institut Fourier},
     volume = {62},
     number = {6},
     %doi = {10.5802/aif.2746},
     %mrnumber = {3060755},
     %zbl = {1272.14016},
     %language = {en},
     %url = {https://aif.centre-mersenne.org/articles/10.5802/aif.2746/}
}

@article{generalCI,
      title={K-stability of complete intersection}, 
      author={Livia Campo and Kristin DeVleming and Tiago Duarte Guerreiro},
      year={2026},
      journal={arXiv},
      number={2608.09293},
      archivePrefix={arXiv},
      primaryClass={math.AG},
      %url={https://arxiv.org/abs/2601.02974}, 
}

@article{LiuXucubicthreefold,
  title={K-stability of cubic threefolds},
  author={Yuchen Liu and Chenyang Xu},
  journal={Duke Mathematical Journal},
  year={2017},
  %url={https://api.semanticscholar.org/CorpusID:119329626}
}

@article{LiuCubicFourfold,
  title={K-stability of cubic fourfolds},
  author={Liu, Yuchen},
  journal={Journal f{\"u}r die reine und angewandte Mathematik (Crelles Journal)},
  volume={2022},
  number={786},
  pages={55-77},
  year={2022},
  publisher={De Gruyter}
}

@article{BenoistCompleteIntersection,
  title={Quelques espaces de modules d'intersections compl{\`e}tes lisses qui sont quasi-projectifs.},
  author={Benoist, Olivier},
  journal={Journal of the European Mathematical Society (EMS Publishing)},
  volume={16},
  number={8},
  year={2014},
  pages={1749-1774}
}

@article{ADL2023QuarticK3Surfaces,
  title = {K-stability and birational models of moduli of quartic K3 surfaces},
  author = {Ascher, Kenneth and DeVleming, Kristin and Liu, Yuchen},
  journal = {Inventiones mathematicae},
  volume = {232},
  number = {2},
  pages = {471--552},
  year = {2023},
  publisher = {Springer},
}

\end{document}